\documentclass[a4paper, 12pt]{amsart}
\usepackage{fullpage} 
\usepackage{tikz} 
\usepackage{amsmath}
\usepackage{hyperref}
\usepackage{tcolorbox}
\usepackage{amsfonts, amsmath, amssymb, amsthm, hyperref, enumitem, fancyhdr, color, soul, comment, standalone}
\usepackage{hyperref}
\hypersetup{
    colorlinks,
    citecolor=black,
    filecolor=black,
    linkcolor=black,
    urlcolor=black
}
\newtheorem{lemma}{Lemma}
\newtheorem{theorem}{Theorem}
\newtheorem{remark}{Remark}
\newtheorem{corollary}{Corollary}

\title{non-linear Smoothing for the Periodic Generalized non-linear Schr\"odinger Equation}
\author{Ryan McConnell}
\address{Department of Mathematics, University of Illinois, Urbana, IL 61801, USA}
\begin{document}
\begin{abstract}
    We consider the periodic non-linear Schr\"odinger equation with non-linearity given by $|u|^{p-1}u$ for odd $p > 1$ in dimension $1$. We first establish that the difference between the non-linear evolution and a phase rotation of the the linear evolution is in a smoother space. We then study forced and damped defocusing non-linear Schr\"odinger equations of the above type and establish an analogous smoothing statement that extends globally in time. As a corollary we establish both existence and smootheness for global attractors in the energy space.
\end{abstract}
\maketitle
\section{Introduction}

We study the non-linear smoothing properties of the periodic p-NLS for odd $p$ given by
\begin{align}\tag{p-NLS}\label{p-NLS}
    \begin{cases}
        iu_t +\bigtriangleup u \pm |u|^{p-1}u = 0\\
        u(x,0) = u_0\in H^s_x(\mathbb{T}),
    \end{cases}
\end{align}
with scaling given by
\[
s_c := \frac{1}{2} - \frac{2}{p-1} = \frac{p-5}{2(p-1)},
\]
when $d = 1$. It's classically known that the cubic ($p = 3$) NLS is locally well-posed for $s\geq s_c$. When $p \geq 5$, $\varepsilon$-losses in the Strichartz estimates give only local well-posedness depending on the $H^s_x(\mathbb{T})$ norm in the region $s > s_c$, \cite{bourgain1993fourier}, whereas $s = s_c$ is a different beast altogether. Indeed, it's known that for the quintic NLS the $L^2_x(\mathbb{T})$ data-to-solution map fails to be analytic, e.g. \cite{kishimoto2014remark}. Similarly, on the real line we have the same ranges, but it's known that when the equation is mass or energy-critical (that is, $s_c = 0$ or $s_c = 1$), then the equation is locally well-posed, but the time of existence depends also on the profile of the initial data, see \cite{cazenave1989some, cazenave1990cauchy}.

Solutions $u$ to \eqref{p-NLS} emanating from $u_0$ come with conserved mass
\[
\int_\mathbb{T} |u(x,t)|^2\,dx 
\]
as well as conserved energy
\begin{equation}\label{energy}
 \frac{1}{2}\int_\mathbb{T}|\partial_x u(x,t)|^2\,dx \mp \frac{1}{p+1}\int_\mathbb{T}|u(x,t)|^{p+1}\,dx,
\end{equation}
from which global well-posedness for the cubic NLS for $s\geq 0$, defocusing (minus sign) quintic NLS for $s\geq 1$, and focusing (plus sign) quintic NLS for $s\geq 1$ and small datum follows. It's worth noting that when the equation is focusing then we can only use the energy to control the $\dot{H}^1_x(\mathbb{T})$ when the initial datum is smaller than the size of the ground state, and, in fact, there are solutions to the focusing problem at the $H^1_x$ level that exhibit blow up for both $\mathbb{R}$ and $\mathbb{T}$, \cite{fan2017log, oh2012blowup}.

Global well-posedness below the $H^1_x(\mathbb{T})$ level for the quintic NLS was frequently studied in the early $2000's$ following Bourgain's high-low decomposition for the cubic problem on $\mathbb{R}^2$, \cite{bourgain1998refinements}. In particular, using the I-method and a normal form approach, Bourgain \cite{bourgain2004remark} was able to prove that the defocusing quintic NLS was globally well-posed for $s > s^*$ and some $s_* < \frac{1}{2}$. Following this, De Silva et al., \cite{de2007global}, and Li et al.,  \cite{li2011global}, were able to establish global well posedness for the defocusing quintic NLS for $s > \frac{4}{9}$ and $s > \frac{2}{5},$ respectively, heavily using refined bilinear Strichartz estimates that mimic the real line smoothing behavior of the Schr\"odinger group. This is all in marked contrast to the real line, where Dodson \cite{dodson2016global} proved global well-posedness for the defocusing quintic NLS in $L^2_x(\mathbb{R})$ using concentration-compactness methods alongside long-time Strichartz estimates.

On the real line, the smoothing behavior used in \cite{de2007global, li2011global} takes the form
\[
\| W_t\phi(x) W_t\psi(x)\|_{L^2_{x,t}(\mathbb{R}^2)}\lesssim \frac{1}{N_1^{1/2}}\|\phi\|_{L^2_x(\mathbb{R})}\|\psi\|_{L^2_x(\mathbb{R})}
\]
when $\phi$ is supported on frequencies $\xi\sim N_1$, $\psi$ on frequencies $\xi\sim N_2$, $N_1\not\sim N_2$, and $W_t$ is the propagator of the Schr\"odinger group. By considering quadratic level sets (e.g.  \cite{erdougan2016dispersive} Corollary 3.17) we then immediately obtain
\[
\|uv\|_{X^{s,b}_T}\lesssim \frac{1}{N_1^{1/2}}\|u\|_{X^{s,b}_T}\|v\|_{X^{s,b}_T},
\]
for $u(x,t), v(x,t)$ with the same frequency restrictions as a $\phi$ and $\psi$, $b > 1/2$, and $X^{s,b}$ defined in the next section.
The effect of these is then seen to be the ability to move derivatives from a high frequency term to the lower frequency term. With this observation, it then can be proved from the bilinear estimate that for the cubic NLS with $s\geq 0$ we have
\[
\|u-W_tu_0\|_{X^{s+\varepsilon(s),b}_T}\lesssim \|u\|_{X^{s,b}_T}^3,
\]
for 
\[
\varepsilon(s) < \min(2s, 1).
\]

In fact, this phenomenon -- non-linear smoothing --
was fundamental to Bourgain's  high-low method,  \cite{bourgain1998refinements}, and is of interest in periodic problems due to the lack of hallmark dispersive traits. Indeed, through the use of differentiation-by-parts, Erdogan and Tzirakis were able to establish non-linear smoothing for the periodic KdV (see: \cite{erdougan2013global, erdogan2011long}) as well as for the fractional cubic NLS and the quintic NLS on the real line with Gurel, \cite{erdogan2017smoothing}. Furthermore, Oh and Stefanov, \cite{oh2020smoothing}, established smoothing for the generalized KdV using a resonant decomposition and a normal form transformation similar in spirit to the differentiation-by-parts method employed by Erdogan and Tzirakis, with the method generaliezed to the forced and damped equation in \cite{mcconnell2021global}. Such a differentation-by-parts method was also used to show smoothing for the periodic dNLS, \cite{isom2020growth}. For more information about real line smoothing see, e.g. \cite{ correia2020nonlinear, erdogan2017smoothing, keraani2009smoothing, kappeler2017scattering}. In particular, \cite{kappeler2017scattering} establishes non-linear smoothing for the defocusing cubic NLS using complete integrability methods with \textit{no} growth in the smoothing bound.

The flexibility of the method employed in \cite{oh2020smoothing} hints that a similar wide-sweeping statement ought to be possible for the p-NLS family, and is exactly what this paper sets out to do. In particular, through a gauge transformation, careful analysis of the resonant set given in Lemma \ref{General NLS Decomp}, and a normal form reduction similar in spirit to the one used in \cite{oh2020smoothing}, we establish a smoothing result for solutions to \eqref{p-NLS} on $\mathbb{T}$. 
\begin{theorem}\label{Theorem 1: Smoothing}
Let $p\geq 5$ be odd and $u$ a solution to \eqref{p-NLS} on $[-T, T]$ emanating from $u_0\in H^s(\mathbb{T})$ for $s > \frac{p-3}{2(p-1)}$, then there is $\tilde{T}$ so that on $[-\tilde{T}, \tilde{T}]$ $u$ has local non-linear smoothing of order
\[
0 < \varepsilon < \min\left((p-1)s - \frac{p-3}{2}, s - \frac{p-5}{2(p-1)}, 1\right).
\]
That is, 
\[
u - e^{it\bigtriangleup_\mathbb{T}}e^{\pm i\frac{p+1}{4\pi}\int_{0}^t\|u(x,s)\|_{L^{p-1}_x(\mathbb{T})}^{p-1}\,ds}u_0\in C_t^0([-\tilde{T}, \tilde{T}];H_x^{s+\varepsilon}(\mathbb{T})).
\]
\end{theorem}
It should be noted that the portion removed is a phase shift of the free solution, and the exponent is well-defined by the standard Strichartz estimates of the next section. When one considers the cubic NLS (that is, $p=3$), then this quantity reduces to a phase rotation by a multiple of the $L^2_x$ norm of $u$-- which is exactly $\|u_0\|_{L^2_x(\mathbb{T})}^2$ by conservation of the $L^2$ norm. For more on the cubic NLS see \cite{erdogan2013talbot, kappeler2017scattering}.
\begin{remark}
The level at which the smoothing begins is exactly the local well-posedness level for the $(2p-1)-$NLS, which is rather far from the local well-posedness level. Notice, however, that for the quintic NLS, Theorem \ref{Theorem 1: Smoothing} has smoothing beginning at $s = \frac{1}{4}$, which is still lower than the currently known global well-posedness level of $s > \frac{2}{5}$. 
\end{remark}
\begin{remark}
Theorem \ref{Theorem 1: Smoothing} doesn't see the focusing or defocusing nature of the equation, and these properties are only necessary to establish that the smoothing effect is global at the $H^1(\mathbb{T})$ level under the assumption that the energy controls the $\dot{H}^1_x(\mathbb{T})$ norm.
\end{remark}

\begin{remark}
Isom, Mantzavinos, and Stefanov, \cite{isom2020growth}, recently showed smoothing for the dNLS. Implicit in their proof is smoothing for the quintic NLS for $s > \frac{1}{2}$ and $\varepsilon < \frac{1}{2}$. The proof employed here is similar to their Lemma 3.5, but with the added benefit of the normal form transformation allowing the smoothing to continue until $\varepsilon < 1$.
\end{remark}

As a corollary, we study the forced and damped defocusing \eqref{p-NLS}
\begin{align}\tag{p-fdNLS}\label{p-fdNLS}
    \begin{cases}
        iu_t +\bigtriangleup u - |u|^{p-1}u +i\gamma u= f(x)\\
        u(x,0) = u_0\in H^s_x(\mathbb{T}),
    \end{cases}
\end{align}
for time-independent $f$, $s \geq 0$, and $\gamma > 0$. The author is unaware of results in the periodic setting for generalized NLS equations with higher order non-linearities $(p > 3)$ in the spirit of \cite{ghidaglia1988finite, goubet2000asymptotic}.

With that said, it's known that for the cubic non-linearity, under the assumption that $f\in L^2_x(\mathbb{T})$, the system possesses a global attractor in $L^2_x(\mathbb{T})$ that is compact in $H^2_x(\mathbb{T})$, \cite{molinet2009global}, and, furthermore, \cite{wang1995energy} establishes the existence of a global attractor in $H^1_x(\mathbb{T})$ and $H^2_x(\mathbb{T})$, upgrading a result of \cite{ghidaglia1988finite}. Similarly, it's known, \cite{goubet1996regularity}, that the weakly damped cubic NLS with smooth forcing has a global attractor in $C^\infty_x(\mathbb{T})$ with an analogous statement on $\mathbb{R}$ in \cite{akroune1999regularity} proving the compactness of the $H^1_x(\mathbb{R})$ global attractor in $H^2_x(\mathbb{R})$. There is more known on the real line, where several recent results by Goubet \cite{goubet2017finite, goubet2020global} showed the existence of a global attractor in $H^\alpha_x(\mathbb{R})$ for the cubic fractional NLS with spatial derivative term $(-\bigtriangleup)^{\alpha}u$ for $\alpha\in(-\frac{1}{2}, 1]$ and in a subset of $H^1_x(\mathbb{R})$ for the NLS with non-linearity $u(\log|u|)^2.$ Of additional interest is a result of Tao \cite{tao2008global}, where the existence of a global attractor in the subset of spherically symmetric functions in $H^1_x(\mathbb{R}^d)$ for the mass super-critical and energy sub-critical \eqref{p-NLS} is shown, merely requiring a potential term $V\in C^\infty_0(\mathbb{R}^d)$ and $d$ sufficiently large. 

With smoothing in mind, we show, using the method of \cite{ erdogan2011long, erdougan2013smoothing}, the existence of a global attractor in the energy space for \eqref{p-fdNLS}.
\begin{corollary}\label{Corollary: global Attractor}
Let $u_0\in H^1_x(\mathbb{T})$, $f\in H^{1}_x(\mathbb{T})$ time independent, $p\geq 5$ odd, $\gamma > 0$. Then the solution to \eqref{p-fdNLS} emanating from $u_0$ extends globally. Furthermore, \eqref{p-fdNLS} possesses a global attractor $\mathcal{A}\subset H^1_x$ that is compact in $H^s_x$ for $s\in \left(1, \frac{3p+1}{2(p-1)}\right).$
\end{corollary}

\begin{remark}
Goubet \cite{goubet2000asymptotic} showed the existence of a global attractor in $H^1(\mathbb{T}^2)$ for general sub-cubic non-linearities. As the quintic NLS on $\mathbb{T}$ is nearly the cubic NLS on $\mathbb{T}^2$, we find the above result obvious. The main novelty of Corollary \ref{Corollary: global Attractor} is then the ease of the method, and the generality.
\end{remark}

The organization of the paper is as follows: in Section \ref{Section: Resonant Decomp} we first perform a resonant decomposition to understand the underlying phase space and the multilinear operators better; in Section \ref{Section: Normal Form Reduction} we then use the resonant decomposition to manipulate the equation and apply a normal form transformation in order to achieve smoothing; in Section \ref{Section: Global Attractors} we then slightly extend the smoothing to \ref{p-fdNLS} in order to establish the existence of a global attractor. 
\section{Smoothing}
\subsection{Background}
Let $k$ be the dual Fourier variable to $x\in\mathbb{T}: = \mathbb{R}/(2\pi\mathbb{Z})$ and $\tau$ be the dual Fourier variable to $t\in\mathbb{R}$. We denote the Fourier transform of a function $u$ on the torus $\mathbb{T}$ by $\widehat{u}_{k}$. When $u$ is a spacetime function on $\mathbb{T}\times\mathbb{R}$, we denote both the space-time Fourier transform and the spacial Fourier transform as $\widehat{u}_k$, since confusion is mitigated by including the variables. 

We write $A\lesssim_{\epsilon} B$ when there is a constant $0 < C(\epsilon)$ such that $A\leq CB$; $A\gg_\epsilon B$ to be the negation of $A\lesssim_\epsilon B$; and $A\sim_\epsilon B$ if, in addition, $B\lesssim_\epsilon A$. We will also write $g = O(f)$ to denote $|g|\lesssim |f|$ and $a+$  (resp. $a-$) to denote $a+\epsilon$ (resp. $a-\epsilon$) for all $\epsilon > 0$, with implicit constants depending on $\epsilon$. It will be understood that the implicit constants will depend on the smoothing parameter used in Section \ref{Lemmas} (denoted as $\varepsilon$).

Define $\langle h\rangle := (1+|h|^2)^{1/2}$ and $J^s$ to be the Fourier multiplier given by $\langle k\rangle^s$. Let $W_t = e^{it\partial_x^2}$ be the propagator for the Schr\"odinger equation, and for clarity we ignore factors of $2\pi$. We then define the Bourgain space $X^{s,b}$ first introduced in the seminal paper \cite{bourgain1993fourier}, by
\[
    \|u\|_{X^{s,b}} = \|\langle k \rangle^{s}\langle \tau + k^2\rangle^{b}\widehat{u}\|_{L^2_\tau\ell^2_k} = \|\langle k\rangle ^s\langle \tau\rangle^b \widehat{u}(k, \tau - k^2\rangle)\|_{\ell^2_k L^2_\tau}=\|W_{-t}u\|_{H^s_xH^b_t},
\]
which measures, in a sense, how much $u$ deviates from the free solution $W_tu_0$. Given $T > 0$ we define the restricted Bourgain space as the space of equivalence classes of functions endowed with the norm
\begin{equation*}
\|u\|_{X^{s,b}_T} := \inf\{\|v\|_{X^{s,b}}\,|\,u|_{[0,T ]} = v|_{[0,T]}\},
\end{equation*}
with dual space $X^{-s,-b}.$ Because of this natural pairing we will often invoke duality, where it will be natural to define the hyper-plane for odd $p$ by
\[
\Gamma_p := \{\tau-\tau_1 +\cdots -\tau_p = 0,\,\,k-k_1+\cdots -k_p = 0\},
\]
with obvious inherited measure denoted by $d\Gamma$. Additionally, we will often ignore intricacies with conjugates, as the modifications to the Strichartz estimates are, by now, standard.

We now record some facts (see: \cite{erdougan2016dispersive}):
\begin{lemma}
For any $\chi\in \mathcal{S}(\mathbb{R})$, $f\in C^\infty_x(\mathbb{T})$, $F\in X^{s,-b}$, and $b > 1/2$,
\begin{align*}
    \|\chi(t)e^{it\partial_x^2}f\|_{X^{s,b}}&\lesssim \|f\|_{H^s_x}\\
    \left\|\chi(t)\int_0^te^{i(t-s)\partial_x^2}F(s)\,ds\right\|_{X^{s,b}}&\lesssim \|F(s)\|_{X^{s, -b}}.
\end{align*}
\end{lemma}
Additionally, the standard Strichartz estimate for the Schr\"odinger equation
\[
\|W_tu_0\|_{L^q_x(\mathbb{T}^2)}\lesssim \|u_0\|_{H^s_x(\mathbb{T})}\qquad\qquad s > \frac{1}{2}-\frac{3}{q}\qquad\qquad q>6
\]
implies the $X^{s,b}$ variant
\[
\|\chi(t/T)u\|_{L^q_x(\mathbb{T}\times\mathbb{R})}\lesssim \|u\|_{X^{s,b}_T},
\]
for $b > 1/2$ and $s$ in the same range as for the Strichartz estimate. 

Let $N$ be dyadic, $P_N$ denote the standard Littlewood-Paley Fourier multiplier onto the frequencies $k\sim N$, and $u_i = P_{N_i}u_i$ for dyadic $N_i$. By a Cauchy-Schwarz and a standard multilinear interpolation argument (e.g. see (3.7) of \cite{wang2013periodic}, and then multilinear interpolation as in \cite{erdougan2016dispersive}), we have the following multilinear Strichartz estimate: given $\varepsilon > 0$ there is $b' < 1/2$ so that 
\begin{equation}\label{Multilinear Strichartz}
\left\|\prod_{i=1}^{\frac{p+1}{2}}\chi(t/T)u_i\right\|_{L^2_{x,t}}\lesssim \|u_1\|_{X^{0,b'}_T}\prod_{i=2}^\frac{p+1}{2}N_i^{\frac{p-5}{2(p-1)}+\varepsilon}\|u_i\|_{X^{0,b'}_T},
\end{equation}
for odd $p\geq 5$. While this is strictly worse than the estimate presented in \cite{wang2013periodic}, the losses only lose the endpoints, which we do not care about.

\subsection{Resonant Decomposition}\label{Section: Resonant Decomp}
We begin by giving a resonant set decomposition in the spirit of \cite{oh2020smoothing} for \eqref{p-NLS}. Define $\Phi$ to be
\begin{align}\tag{$\Phi$}
    \Phi(k,k_1, \cdots, k_p) &= k^2-k_1^2+k_2^2-\cdots-k_p^2,\\
    k&=k_1-k_2+\cdots+k_p.\nonumber
\end{align}

The quantity $\Phi$ frequently appears when performing multilinear estimates using Cauchy Schwarz. Indeed, if $T$ is a multilinear operator given by

\[
\widehat{T(u^1, \cdots, u^p)}(k) = \sum_{k=k_1-k_2+\cdots +k_p}m(k_1, \cdots, k_n) \prod_{\substack{\ell=1\\odd}}^p\widehat{u^\ell}_{k_\ell}\prod_{\substack{\ell=2\\even}}^p\overline{\widehat{u^\ell}_{k_\ell}}
\]
and 
\begin{align*}
    \Gamma^{k,\tau}_p &:=\{\tau = \tau_1-\tau_2+\cdots+\tau_p,\,\,k=k_1-k_2+\cdots k_p\},
\end{align*}
then by suppressing $\tau$ dependence, invoking duality for $v\in X^{-s,b'}$, and defining
\begin{align*}
    \overline{\widehat{w}}_k := \langle k \rangle ^{-s}\langle \tau-k^2\rangle ^{b'}\overline{\widehat{v}}_k, \qquad\qquad \widehat{z^\ell}_{k_\ell} := \langle k_\ell\rangle^{s}\langle \tau_\ell +(-1)^{\ell+1} k_\ell^2\rangle^{b}\widehat{u^\ell}_{k_\ell},
\end{align*}
we have
\begin{align*}
    &\|T(u^1, \cdots, u^p)\|_{X^{s,-b'}} = \sup_{\|w\|_{L^2_{x,t}}=1} \bigg|\underset{\Gamma_p^{k,\tau}}{\int\sum}\frac{m(k_1, \cdots, k_p)\langle k\rangle^s\langle \tau - k^2\rangle^{-b'}}{\prod_{\substack{\ell=1\\odd}}^p\langle k_\ell\rangle^s\langle \tau + k_\ell^2\rangle^b\prod_{\substack{\ell = 2\\even}}^p\langle k_\ell\rangle^s\langle \tau - k_\ell^2\rangle^b}\overline{\widehat{w}_k}\\
    &\qquad\qquad\times\prod_{\substack{\ell=1\\odd}}^p\widehat{z^\ell}_{k_\ell}\prod_{\substack{\ell=2\\even}}^p\overline{\widehat{z^\ell}_{k_\ell}}\,d\Gamma\bigg|\\
    &\leq \sup_{\|w\|_{L^2_{x,t}}=1}\underset{\tau,k}{\int\sum}\left(\underset{\Gamma_p^{k,\tau}}{\int\sum}\frac{m^2(k_1, \cdots, k_p)\langle k\rangle^{2s}\langle \tau - k^2\rangle^{-2b'}}{\prod_{\substack{\ell=1\\odd}}^p\langle k_\ell\rangle^{2s}\langle \tau + k_\ell^2\rangle^{2b}\prod_{\substack{\ell = 2\\even}}^p\langle k_\ell\rangle^{2s}\langle \tau - k_\ell^2\rangle^{2b}}\,d\tau_1\,\cdots d\tau_{p-1}\right)^{1/2}\\
    &\qquad\qquad \times\left(\underset{\Gamma_p^{k,\tau}}{\int\sum}\left|\prod_{\substack{\ell=1\\odd}}^p\widehat{v^\ell}_{k_\ell}\prod_{\substack{\ell=2\\even}}^p\overline{\widehat{v^\ell}_{k_\ell}}\right|^2\,d\tau_1\,\cdots\,d\tau_{p-1}\right)^{1/2}|\overline{\widehat{w}_k}|\,d\tau\\
    &\leq \sup_{\tau, k}\left(\underset{\Gamma_p^{k,\tau}}{\int\sum}\frac{m^2(k_1, \cdots, k_p)\langle k\rangle^{2s}\langle \tau - k^2\rangle^{-2b'}}{\prod_{\substack{\ell=1\\odd}}^p\langle k_\ell\rangle^{2s}\langle \tau + k_\ell^2\rangle^{2b}\prod_{\substack{\ell = 2\\even}}^p\langle k_\ell\rangle^{2s}\langle \tau - k_\ell^2\rangle^{2b}}\,d\tau_1\,\cdots d\tau_{p-1}\right)^{1/2}\\
    &\qquad\qquad\times\prod_{\ell=1}^p\|u^\ell\|_{X^{s,b}}.
\end{align*}
An application of (\cite{erdogan2017smoothing}, Lemma 2.1) relates the supremum term to $\Phi$ for $0 < b' < 1/2$, $b > 1/2$:
\begin{align}
\sup_{\tau, k}&\left(\underset{\Gamma_p^{k,\tau}}{\int\sum}\frac{m^2(k_1, \cdots, k_p)\langle k\rangle^{2s}\langle \tau - k^2\rangle^{-2b'}}{\prod_{\substack{\ell=1\\odd}}^p\langle k_\ell\rangle^{2s}\langle \tau + k_\ell^2\rangle^{2b}\prod_{\substack{\ell = 2\\even}}^p\langle k_\ell\rangle^{2s}\langle \tau - k_\ell^2\rangle^{2b}}\,d\tau_1\,\cdots d\tau_{p-1}\right)^{1/2}\nonumber\\
&\lesssim \sup_{\tau, k}\left(\underset{k = k_1-k_2+\cdots +k_p}{\sum}\frac{m^2(k_1, \cdots, k_p)\langle k\rangle^{2s}}{\langle \Phi\rangle^{2b'}\prod_{\substack{\ell=1}}^p\langle k_\ell\rangle^{2s}}\right)^{1/2}.\label{CS Application Bound}
\end{align}
Because of the algebraic relationship between the $k_\ell$'s, we will have at least one $k_\ell$ to cancel the $\langle k\rangle^{2s}$ factor, and hence we can tease additional regularity if we have either good lower bounds on $\Phi$, or knowledge about additional large frequencies. Note that below the Sobolev embedding level we will often have to use knowledge about $\Phi$ to simply conclude that the above sum converges, and that technicality will be a factor in the lower bound for the smoothing level\footnote{As well as the fact that we use a differentiation-by-parts scheme that increases the size of the non-linearity. Since we use the derivative gained for smoothing, we lose the ability to use the derivative gain to lower the regularity in the multilinear estimates.}.

The next lemma is exactly the tool we will use to establish either large factors, or a good lower bound for $\Phi$. The proof is elementary and is similar to Lemma 3.5 of \cite{isom2020growth} and the decomposition done in \cite{oh2020smoothing}. For the purposes of the next lemma and the remainder of the paper, let the magnitude of the $\ell$'th largest frequency be denoted by $k^*_\ell$-- that is, $k_\ell^*$ denotes the decreasing rearrangement of $|k_1|, \cdots, |k_p|.$ Furthermore, if $k_1,\cdots, k_p\in\mathbb{Z}$, then we call a configuration $(k_1, \cdots, k_p)$ \textit{resonant} when there is an odd $\ell$ with 
\begin{equation*}
k_{\ell} = k_1 - k_2 + \cdots + k_p.
\end{equation*}
\begin{lemma}\label{General NLS Decomp}
Let $p\geq 5$ be odd and fix $k, k_1, \cdots, k_p$ so that 
\[
k = k_1-k_2+\cdots+k_p.
\] 
Then at least one of the following must be true
\begin{itemize}
    \item[A)] There is exactly one odd $\ell$ such that $k_\ell = k$;
    \item[B)] $|\Phi|\gtrsim k_1^*$.
    \item[C)] $(k_3^*)^2\gtrsim k^*_1$ or $k_1^*\sim k_2^*$.
\end{itemize}
\end{lemma}
\begin{proof}
Organize the frequencies so that $k_1\geq k_3\cdots\geq k_p$ and similarly for the evens. We first note that if we are not in Case $A$ and there is resonance, then by the algebraic relationship among $k, k_1,\cdots, k_p$ we must have $k_3^*\gtrsim k$. In the situation that $k_1^*\sim k$, we find that $k_3^*\gtrsim k_1^*$ and in the situation $k_1^*\gg k$, we find that $k_1^*\sim k_2^*$. In either situation we must live in Case $C$.

Suppose now that we're not in Cases $A$ or $C$, so that we must either have $(k_3^*)^2, (k_2^*)^2\ll k_1^*,$ or $(k_3^*)^2\ll k_1^*\lesssim (k_2^*)^2$. Then we first assume $(k_2^*)^2\ll k_1^*$, and $k_1^* = k_1$, so that:
\[
|\Phi| \geq |k^2-k_1^2|-p(k_2^*)^2= |k-k_1||k+k_1|-p|k_2^*|^2\gtrsim |k_1|.
\]
Notice that the above holds because $|k-k_1|=0$ only if there is a If $k_1^* = k_2$ then the proof follows similarly from 

\begin{equation}\label{Equation: Phi big when even big}
|\Phi| \geq k^2+k_2^2-p|k_2^*|^2\gtrsim k_2^2.
\end{equation}
Now, if $(k_3^*)^2\ll k_1^*\lesssim (k_2^*)^2$, $k_1^* = k_1$, and $k_2^*=k_2$, then
\begin{align*}
|\Phi|&\geq |(k_1-k_2+\cdots+k_p)^2-k_1^2+k_2^2|-p(k_3^*)^2\\
&\geq 2|k_2^2-k_1k_2|-p(k_3^*)^2-pk_1^*k_3^*\\
&\gtrsim (k_2^*)^2\gtrsim k_1^*,
\end{align*}
with a similar result for the other choices of $k_1^*$ and $k_2^*$. Note that we have used that $k_1^*\gg k_2^*$ in two locations: to conclude that $k-k_1\ne 0$ if there isn't a single resonance, and to conclude that $|k_1-k_2|\gtrsim k_1$ in the penultimate step.
\end{proof}

The remainder of the paper will utilize the above lemma by first singling out Case A, and then singling out Case B from what remains. That is, our decomposition will, for the non-resonant multi-linear operators, be a splitting of the frequency space into $A\sqcup (B\setminus A)\sqcup (C\setminus (B\cup A)),$ where $A, B, C$ represent subsets of $(\mathbb{Z}\setminus\{0\})^p$ satisfying the analogous cases above.

\begin{remark}
Case $A$ is phrased differently than the definition of resonance. This is because we are only able to remove Case $A$ from the resonant situation.
\end{remark}
\subsection{Normal Form Reduction}\label{Section: Normal Form Reduction}
In this section we use Lemma \ref{General NLS Decomp} to prove smoothing for p-NLS. We first start by removing a portion of the resonant set (Case A) so that we may relegate what remains to to Case $C$. Define \textit{Resonance}, denoted $\mathcal{R}$, to contain the resonant tuples, and \textit{Non-resonant}, denoted $\mathcal{NR}$, to contain the remaining terms. We decompose $\mathcal{R}$ into $\mathcal{R}^1$ and $\mathcal{R}^2$, where
\begin{equation}\label{Definition: R1}
\widehat{\mathcal{R}^1}[u, \cdots, u]_k := \sum_{\substack{\ell=1\\odd}}^{p} \sum_{\substack{k=k_1-k_2+\cdots-k_{p-1} +k_p\\k_\ell = k}}\prod_{\substack{i=1\\odd}}^{p}\widehat{u}_{k_i}\prod_{\substack{i=2\\even}}^p\overline{\widehat{u}_{k_i}} = \frac{p+1}{4\pi}\widehat{u}_k\int_\mathbb{T}|u|^{p-1}\,dx.
\end{equation}
The principle of inclusion-exclusion then gives a nice interpretation for $\mathcal{R}^2$ as containing the Resonant set with at least two odd-indexed internal frequencies that are equal to the external frequency. By the algebraic relation this forces there to be at least $3$ internal frequencies $\gtrsim |k|$.

Define the Fourier multiplier:
\begin{equation}\label{L definition}
L_t[u] = \exp\left(\mp i\frac{p+1}{4\pi}\int_0^t\int_\mathbb{T} |u|^{p-1}(x,s)\,dxds\right),
\end{equation}
whose effect will be removing the $\mathcal{R}^1$ portion of $\mathcal{R}$. 

Let $v = \mathcal{F}^{-1}\left(L_t[u]\widehat{u}\right)$, and notice that for odd $p$ we have
\begin{equation}\label{Multiplier Fact: Preserve Mean}
\int_\mathbb{T}|u|^{p-1}(x,t)\,dx = 2\pi\sum_{0=k_1-k_2+\cdots -k_{p-1}}\widehat{u}_{k_1}\overline{\widehat{u}}_{k_2}\cdots\widehat{u}_{k_{p-2}}\overline{\widehat{u}}_{k_{p-1}} = \int_\mathbb{T}|v|^{p-1}(x,t)\,dx,
\end{equation}
and
\begin{equation}\label{Multiplier fact: Nonlinearity}
\widehat{\left(|u|^{p-1}u\right)}_k = 2\pi\sum_{k=k_1-k_2+\cdots +k_{p}}\widehat{u}_{k_1}\overline{\widehat{u}}_{k_2}\cdots\widehat{u}_{k_{p-2}}\overline{\widehat{u}}_{k_{p-1}}\widehat{u}_{k_p} = \frac{1}{L_t[u]}\widehat{\left(|v|^{p-1}v\right)}_k.
\end{equation}
Now, applying \eqref{Multiplier Fact: Preserve Mean}, \eqref{Definition: R1}, and the fact that $L_t$ is independent of $x$, we find
\[
i\widehat{u}_t = \left\{\mp\frac{p+1}{4\pi}\int_\mathbb{T}|u|^{p-1}(x,t)\,dx\right\}\frac{\widehat{v}}{L_t[u]} + i\frac{\widehat{v}_t}{L_t[u]} = \mp\widehat{\mathcal{R}^1}[v, \cdots, v]\frac{1}{L_t[u]} + i\frac{\widehat{v}_t}{L_t[u]},
\]
so that our new equation becomes
\begin{align}\label{v equation}
    \begin{cases}
        iv_t +\bigtriangleup v \pm \mathcal{R}^2[v] \pm \mathcal{NR}[v] = 0\\
        v(x,0) = v_0 = u_0\in H^s_x(\mathbb{T}).
    \end{cases}
\end{align}
This equation has similar well-posedness properties as \eqref{p-NLS}, but since the multiplier is time dependent we don't know a priori if there is a relationship in $X^{s,b}$ between $u$ and $v$. Because of this and the later use of the local well-posedness bound, we include the following lemma, whose proof is standard.
\begin{lemma}\label{Lemma: v local well posedness}
Let $p\geq 5$ be odd, $s > \frac{p-5}{2(p-1)}$, and $u_0\in H^s_x(\mathbb{T})$. Then there is a $T = T(\|u_0\|_{H^s_x(\mathbb{T})})$ and a unique solution $v\in X^{s,1/2+}_T$ to \eqref{v equation} emanating from $u_0$ that satisfies
\[
\|v\|_{X^{s,1/2+}_T}\lesssim \|u_0\|_{H^s_x(\mathbb{T})}.
\]
\end{lemma}

To remove certain high-low frequency interactions specific to the situation that $|\Phi|\gtrsim |k|$, we consider the portion of $\mathcal{NR}$ given by one large odd internal frequency, in the sense that $\ell$ odd and $|k_\ell|\gg k_2^*$. With this in mind, for odd $\ell$ we define
\begin{equation}
    \widehat{T_\ell}[v_1,\cdots, v_p]_k = \sum_{\substack{k=k_1-k_2+\cdots-k_{p-1} +k_p\\|k_\ell|\gg k_2^*\\ Case\, B}}\frac{\prod_{\substack{j=1\\odd}}^{p}\widehat{v}_{k_i}\prod_{\substack{j=2\\even}}^p\overline{\widehat{v}_{k_i}}}{\Phi(k, k_1\cdots, k_p)},
\end{equation}
and
\begin{equation}\label{defn of T}
    \widehat{T}[W_t u_0, v, \cdots, v]_k = \sum_{\ell\,\,odd}\widehat{T_\ell}[v, \cdots, v, \underset{\ell'th}{\underbrace{W_t u_0}}, v, \cdots, v]_k = \frac{p+1}{2}\widehat{T_1}[W_tu_0, v, \cdots, v]_k,
\end{equation}
where it's understood that Case $B$ is in the language of our decomposition lemma.

The use of this transformation is that, by symmetry, we have:
\begin{align}
\big(i\partial_t-&k^2\big)\widehat{T}_1[W_t u_0,v,\cdots, v]_k\label{Calculation: T in Equation}\\
&= \sum_{\substack{k=k_1-k_2+\cdots-k_{p-1} +k_p\\|k_1|\gg k_2^*\\ case\, B}}\frac{k_1^2-k^2}{\Phi(k, k_1,\cdots, k_p)}\widehat{(W_tu_0)_{k_1}}\prod_{\substack{j=3\\odd}}^{p}\widehat{v}_{k_i}\prod_{\substack{j=2\\even}}^p\overline{\widehat{v}_{k_i}}\nonumber\\
&\qquad+\sum_{\substack{k=k_1-k_2+\cdots-k_{p-1} +k_p\\|k_1|\gg k_2^*\\ case\, B}}\frac{p-1}{2\Phi(k, k_1\cdots, k_p)}\widehat{(W_tu_0)_{k_1}}(i\partial_t\widehat{v}_{k_3})\prod_{\substack{j=5\\odd}}^{p}\widehat{v}_{k_i}\prod_{\substack{j=2\\even}}^p\overline{\widehat{v}_{k_i}}\nonumber\\
&\qquad+\sum_{\substack{k=k_1-k_2+\cdots-k_{p-1} +k_p\\|k_1|\gg k_2^*\\ case\, B}}\frac{p-1}{2\Phi(k, k_1\cdots, k_p)}\widehat{(W_tu_0)_{k_1}}\overline{(-i\partial_t\widehat{v}_{k_2})}\prod_{\substack{j=3\\odd}}^{p}\widehat{v}_{k_i}\prod_{\substack{j=4\\even}}^p\overline{\widehat{v}_{k_i}}\nonumber\\
&=-\sum_{\substack{k=k_1-k_2+\cdots-k_{p-1} +k_p\\|k_1|\gg k_2^*\\ case\, B}}\widehat{(W_tu_0)_{k_1}}\prod_{\substack{j=3\\odd}}^{p}\widehat{v}_{k_i}\prod_{\substack{j=2\\even}}^p\overline{\widehat{v}_{k_i}}\nonumber\\
&\qquad +\sum_{\substack{k=k_1-k_2+\cdots-k_{p-1} +k_p\\|k_1|\gg k_2^*\\ case\, B}}\frac{p-1}{2\Phi(k, k_1\cdots, k_p)}\widehat{(W_tu_0)_{k_1}}(i\partial_t-k_{3}^2) \widehat{v}_{k_3}\prod_{\substack{j=5\\odd}}^{p}\widehat{v}_{k_i}\prod_{\substack{j=2\\even}}^p\overline{\widehat{v}_{k_i}}\nonumber\\
&\qquad-\sum_{\substack{k=k_1-k_2+\cdots-k_{p-1} +k_p\\|k_1|\gg k_2^*\\ case\, B}}\frac{p-1}{2\Phi(k, k_1\cdots, k_p)}\widehat{(W_tu_0)_{k_1}}\overline{(i\partial_t-k_2^2)\widehat{v}_{k_{2}}}\prod_{\substack{j=3\\odd}}^{p}\widehat{v}_{k_i}\prod_{\substack{j=4\\even}}^p\overline{\widehat{v}_{k_i}}\nonumber\\
&=-\frac{2}{p+1}\widehat{\mathcal{HL}_B}[W_tu_0,v, \cdots, v]_k\nonumber\\
&\qquad- \frac{p-1}{2}\widehat{T_1}[W_t u_0, (i\partial_t+\bigtriangleup)v, \cdots, v]_k\nonumber\\
&\qquad+\frac{p-1}{2}\widehat{T_1}[W_t u_0, v, (i\partial_t+\bigtriangleup)v, \cdots, v]_k,\nonumber
\end{align}
where 
\begin{align}
\widehat{\mathcal{R}^2}[v, \cdots, v]+\widehat{\mathcal{NR}}[v, \cdots, v] &= \widehat{\mathcal{E}}[v, \cdots, v]+\widehat{\mathcal{HL}_B}[v, \cdots, v],\label{Definition: E}\\
\widehat{\mathcal{HL}_B}[v, \cdots, v] &:= \frac{p+1}{2}\sum_{\substack{k=k_1-k_2+\cdots-k_{p-1} +k_p\\|k_1|\gg k_2^*\\ Case\, B}}\prod_{\substack{j=1\\odd}}^{p}\widehat{v}_{k_i}\prod_{\substack{j=2\\even}}^p\overline{\widehat{v}_{k_i}}\label{Definition: HL1}.
\end{align}
Note that $\mathcal{HL}_B$ refers to the high-low interaction in Case $B$.

In order to employ the normal form transformation, we recall that $v(x,0) = u_0$, let
\begin{equation}\label{z definition}
v = W_t u_0 \pm T[W_tu_0, v, \cdots, v] + z.
\end{equation}
We then substitute into the linear portion of the equation as well as the high frequency component of $\mathcal{HL}_B$, so that the final equation reads:
\begin{align}\label{Equation}
    \begin{cases}
        iz_t +\bigtriangleup z \pm \mathcal{E}[v, \cdots, v] \pm \mathcal{HL}_B[z\pm T[W_t u_0, v, \cdots, v], v, \cdots, v]\\ 
        \,\,\,\mp  \frac{p-1}{2}T[W_tu_0, v, \left(\mathcal{R}^2+\mathcal{NR}\right)[v, \cdots, v],v, \cdots, v]\\
\,\,\,\,\,\,\pm \frac{p-1}{2}T[ W_tu_0, \left(\mathcal{R}^2+\mathcal{NR}\right)[v,\cdots, v],v, \cdots, v]= 0\\
        z(x,0) = \mp T[u_0, \cdots, u_0]\in H^s_x(\mathbb{T}),
    \end{cases}
\end{align}
where we have omitted repeated entries.

While this normal form transformation doesn't go down to the level of local well-posedness for the p-NLS, we still can prove this for a large range: 
\begin{lemma}\label{Lemma: Normal Form Regularity}
Let $p\geq 5$ be odd, $u_0\in H^s_x$, $s > \frac{p-3}{2(p-1)}$, and $T$ defined for $p$ as above. Then
\[
\|T[u_0, \cdots, u_0]\|_{H^{s+\varepsilon}_x}\lesssim \|u_0\|_{H^s_x}^{p},
\]
for
\[
0 < \varepsilon < \min\left((p-1)s - \frac{p-3}{2}, 1\right).
\]
\end{lemma}
\begin{remark}
The range of smoothing for $T$ in $H^s_x$ and Lemma \ref{Lemma: HH & R2 regularity} motivates the definition of $T$ in terms of odd-indexed entries. In particular, we see that there is no value added by including the even-indexed entries, and we simply complicate the definition.
\end{remark}
\begin{proof}
Note that $\Phi\gtrsim k$ by summation restrictions. By using Cauchy-Schwarz we reduce to having to bound
\begin{align*}
\sup_{k}\sum_{\substack{k = k_1-k_2+\cdots +k_p\\|k_1|\gg k_2^*\\|\Phi|\gtrsim |k|}}\frac{\langle k\rangle^{2s+2\varepsilon}}{|\Phi|^2\prod_{i=1}^p\langle k_i\rangle^{2s}}&\lesssim \sup_{k}\sum_{\substack{k = k_1-k_2+\cdots +k_p\\|k_1|\gg k_2^*\\|\Phi|\gtrsim |k|}}\frac{\langle k\rangle^{2\varepsilon-2}}{\prod_{i=2}^p\langle k_i\rangle^{2s}}\\
&\lesssim \sup_{k}\langle k\rangle ^{2\varepsilon - 2 +(p-1)\max(1-2s, 0)},
\end{align*}
which is finite for
\[
0 < \varepsilon \leq \min\left((p-1)s - \frac{p-3}{2}, 1\right)
\]
and 
\[
s > \frac{p-3}{2(p-1)}.
\]

Note that we have used the fact that $k_2^*\ll |k_1|$ to sum the $p-1$ other variables.
\end{proof}

\subsection{Lemmas}\label{Lemmas}
In this section we will prove smoothing lemmas related to the remaining terms that appear in the Duhamel form of \eqref{Equation}. Note that we keep factors of $T$ in the estimates for clarity as well as for later use in the smoothing proof.

The first lemma concerns the nicest term in \eqref{Equation}.
\begin{lemma}\label{Lemma: HH & R2 regularity}
Let $p\geq 5$ be odd, $s > \frac{p-3}{2(p-1)}$, and $0<T\ll 1$. Then there is $\delta > 0$ so that
\[
\|\mathcal{E}[v, \cdots, v]\|_{X^{s+\varepsilon, -1/2+}_T}\lesssim T^\delta\|v\|_{X^{s, 1/2+}_T}^{p},
\]
for 
\[
0 < \varepsilon < \min\left((p-1)s - \frac{p-3}{2}, s - \frac{p-5}{2(p-1)}, 1\right).
\]
\end{lemma}
\begin{proof} Before proceeding, we again reorganize so that $|k_1|\geq |k_3|\geq\cdots\geq |k_p|$ and similarly for the even indices. 

\textbf{Case A,``Resonance":} Recalling \eqref{Definition: R1}, \eqref{Definition: HL1}, and Lemma \ref{General NLS Decomp}, we see that we have no Case $A$.

\textbf{Case B, } $|\Phi|\gtrsim k_1^*$: By the definition of Case C and the restrictions on  \eqref{Definition: E}, we see that the only terms that live solely in this case remaining after removal are those with $k_2\gg k_2^*$, in which case $|\Phi|\gtrsim (k_1^*)^2$ by \eqref{Equation: Phi big when even big}. By Cauchy-Schwarz, we reduce to bounding
\begin{align}\label{Lemma: HH case 4}
\sup_{k,\tau}\sum_{\substack{|k_2|\gtrsim |k|\\k_2^*\ll |k|\\|\Phi|\gtrsim k^2}} \frac{\langle k\rangle^{2s+2\varepsilon-2+}}{\prod_{i=1}^p\langle k_i\rangle^{2s}}\lesssim \langle k\rangle^{2\varepsilon -2+(p-1)\max(1-2s, 0)+},
\end{align}
which gives smoothing of order
\[
0 < \varepsilon < \min\left((p-1)s-\frac{p-3}{2}, 1\right).
\]
We have again used that $k_2^*\lesssim |k_1|$ in the above.

\textbf{Case C, $(k_3^*)^2\gtrsim k_1^*$ or $k_1^*\sim k_2^*$:} We assume that we have the first situation, as the other will follow in an easier and similar fashion. This case will then follow from the multilinear Strichartz estimate \eqref{Multilinear Strichartz}. Let
\[
\varepsilon = s - \frac{p-5}{2(p-1)}.
\] 
We apply duality with $\eta\in X^{-s-\varepsilon,1/2-}$ to find that we must bound
\begin{align}
    \left|\underset{\Gamma_p}{\int\sum}\overline{\widehat{\eta}}_k\prod_{\substack{j=1\\odd}}^{p}\widehat{v}_{k_j}\prod_{\substack{j=2\\even}}^p\overline{\widehat{v}}_{k_j}\,d\Gamma\right|\label{HH case C bound to show}.
\end{align}
Let $N$ and $N_{k_\ell}$ for $1\leq \ell\leq p$ be dyadic with $|k|\sim N$ and $|k_\ell| \sim N_{\ell}$. We ignore conjugates and assume that $N_1\geq N_2\cdots\geq N_p$. Now, it suffices to show for each collection of fixed $N, N_{1}, \cdots, N_{p}$ satisfying $N_{1}\lesssim N_{3}^2$ that \eqref{HH case C bound to show} is bounded by:
\begin{align*}
N^{-s-\varepsilon-}\|P_{N}\eta\|_{X^{0, 1/2-}}N^{s-}_{1}\|P_{N_{1}}v\|_{X^{0, 1/2-}_T}\prod_{j=2}^3N_{j}^{\frac{p-5}{2(p-1)}+\varepsilon}\|P_{N_{j}}v\|_{X^{0, 1/2+}_T}\prod_{j\geq 4}^p N^{s-}_{j}\|P_{N_{j}}v\|_{X^{0, 1/2+}_T}.
\end{align*}
as we don't care about endpoints. Since we find
\begin{equation}\label{Equation: Case C Dyadic Condition}
N_{1}^{\varepsilon+}\lesssim (N_{3}N_{2})^{\varepsilon+},
\end{equation}
we ignore conjugates and apply Cauchy-Schwarz and \eqref{Multilinear Strichartz} to a single term of the dyadic decomposition of \eqref{HH case C bound to show} for fixed admissible $N, N_{1}, \cdots, N_{p}$ to get
\begin{align*}
    &\lesssim \left\|\prod_{\substack{j=1\\odd}}^pP_{N_{j}}v\right\|_{L^2_{x,t}}\left\|P_{N_k}\eta\prod_{\substack{j=2\\even}}^pP_{N_{j}}v\right\|_{L^2_{x,t}}\\
    &\,\,\lesssim N^{-s-\varepsilon-}\|P_{N}\eta\|_{X^{0, 1/2-}}N_{1}^{s+\varepsilon+}\|P_{N_{1}}v\|_{X^{0, 1/2-}_T}\\
    &\qquad\qquad\times\prod_{j=2}^3N_{j}^{\frac{p-5}{2(p-1)}+}\|P_{N_{j}}v\|_{X^{0, 1/2+}_T}\prod_{j\geq 4}^p N^{s-}_j\|P_{N_{j}}v\|_{X^{0, 1/2+}_T}\\
    &\,\,\lesssim N^{-s-\varepsilon-}\|P_{N}\eta\|_{X^{0, 1/2-}}N^{s-}\|P_{N_{1}}v\|_{X^{0, 1/2-}_T}\\
    &\qquad\qquad\times\prod_{j=2}^3N_{j}^{\varepsilon+\frac{p-5}{2(p-1)}+}\|P_{N_{j}}v\|_{X^{0, 1/2+}_T}\prod_{j\geq 4}^p N^{s-}_j\|P_{N_{j}}v\|_{X^{0, 1/2+}_T}.
\end{align*}
Note that at the final step we use \eqref{Equation: Case C Dyadic Condition} and 
\[
\varepsilon + \frac{p-5}{2(p-1)}+ < s.
\]
It follows that we have smoothing of order
\[
0 < \varepsilon < s - \frac{p-5}{2(p-1)}.
\]

In all cases, we have room for a small power of $T$ by localization.

\end{proof}

Having estimated the nicest of the terms, we now turn to the additional terms created in using the transformation.

\begin{lemma}\label{Lemma: HL[T] Regularity}
Let $p\geq 5$ be odd, $s > \frac{p-3}{2(p-1)}$, and $0 < T\ll 1$. Then there is $\delta > 0$ so that
\[
\|\mathcal{HL}_B[T[W_tu_0, v, \cdots, v], v, \cdots, v]\|_{X^{s+\varepsilon, -1/2+}_T}\lesssim T^\delta\|u_0\|_{H^s_x}\|v\|_{X^{s, 1/2+}_T}^{2p-2},
\]
for 
\[
0 < \varepsilon < 1.
\]
\end{lemma}
\begin{proof}
We have by restrictions on \eqref{defn of T} and \eqref{Definition: HL1} that $k_1\gg k_2^*$. Moreover, we see that
\begin{align*}
&{\mathcal{HL}_B}[T[W_tu_0, v, \cdots, v], v, \cdots, v]_k \\ &\,\,=\frac{p+1}{2}\sum_{\substack{k=k_1-k_2+\cdots-k_{p-1} +k_p\\|k_1|\gg k_2^*\\ Case\, B}}\widehat{T}[W_tu_0, v, \cdots, v]_k\prod_{\substack{j=3\\odd}}^{p}\widehat{v}_{k_j}\prod_{\substack{l=2\\even}}^p\overline{\widehat{v}_{k_l}}\\
&\,\,=\left(\frac{p+1}{2}\right)^2\sum_{\substack{k=k_1-k_2+\cdots-k_{p-1} +k_p\\|k_1|\gg k_2^*\\ Case\, B}}\sum_{\substack{k_1=h_1-h_2+\cdots-h_{p-1} +h_p\\|h_1|\gg h_2^*\\ Case\, B}}\frac{\widehat{(W_tu_0)_{h_1}}\prod_{\substack{r=3\\odd}}^{p}\widehat{v}_{h_r}\prod_{\substack{s=2\\even}}^p\overline{\widehat{v}_{h_s}}}{\Phi(k_1, h_1\cdots, h_p)}\\
&\qquad\qquad\qquad\qquad\qquad\qquad\times\prod_{\substack{j=3\\odd}}^{p}\widehat{v}_{k_j}\prod_{\substack{l=2\\even}}^p\overline{\widehat{v}_{k_l}}\\
\end{align*}
By relabeling variables, noting that $h_1\gtrsim k_1^*,$ we reduce to bounding
\[
\left(\frac{p+1}{2}\right)^2\sum_{\substack{k=k_1-k_2+\cdots-k_{p-1} +k_{2p-1}\\|k_1|\gg k_2^*\\|\Phi(k_1-k_2+\cdots+k_p, k_1\cdots, k_p)|\gtrsim |k_1| }}\frac{\widehat{(W_tu_0)}_{k_1}}{\Phi(k_1-k_2+\cdots+k_p, k_1\cdots, k_p)}\prod_{\substack{j=3\\odd}}^{2p-1}\widehat{v}_{k_i}\prod_{\substack{l=2\\even}}^{2p-1}\overline{\widehat{v}_{k_l}},
\]
which is just a $2p-1$ multilinear operator with a smoothing weight. Since we don't care about the endpoint value of $s$, it follows by duality, ignoring conjugates, and taking $\varepsilon < 1$ so that
\[
\frac{\langle k_1\rangle^{\varepsilon}}{\Phi(k_1-k_2+\cdots+k_p, k_1, \cdots, k_p)}\lesssim 1,
\]
that it suffices to show the bound
\[
\left|\underset{\Gamma_{2p-1}}{\int\sum}\overline{\widehat{\eta}}_k\widehat{(W_tu_0)}_{k_1}\prod_{\substack{j=3\\odd}}^{2p-1}\widehat{v}_{k_j}\prod_{\substack{l=2\\even}}^{2p-1}\overline{\widehat{v}}_{k_l}\,d\Gamma\right|\lesssim \|\eta\|_{X^{-s-\varepsilon,1/2-}}\|u_0\|_{H^s_x}\|v\|^{2p-2}_{X^{s,1/2+}_T}.
\]

Indeed, by the multilinear Strichartz estimate \eqref{Multilinear Strichartz} for $0 < \varepsilon < 1$, we  have the bound
\begin{align*}
\|\mathcal{HL}_B[T[W_tu_0, v, \cdots, v],v, \cdots ,v]\|_{X^{s+\varepsilon, 1/2+}_T} &\lesssim \|W_tu_0\|_{X^{s,1/2-}_T}\|v\|^{2p-2}_{X^{s, 1/2-}_T}\\
&\lesssim T^\delta\|u_0\|_{H^s_x}\|v\|^{2p-2}_{X^{s, 1/2+}_T},
\end{align*}
by localization, given 
\[
s > \frac{2p-6}{2(2p-2)} = \frac{p-3}{2(p-1)}.
\]

\end{proof}

\begin{lemma}\label{Lemma: HL[z] regularity}
Let $p\geq 5$ be odd, $s > \frac{p-5}{2(p-1)}$, and $0 < T\ll 1$. Then there is some $\delta > 0$
\[
\|\mathcal{HL}_B[z, v, \cdots, v]\|_{X^{s+\varepsilon, -1/2+}_T}\lesssim_\delta T^\delta \|z\|_{X^{s+\varepsilon, 1/2+}_T}\|v\|_{X^{s, 1/2+}_T}^{p-1},
\]
for any $\varepsilon \in \mathbb{R}$.
\end{lemma}
\begin{proof}
By a direct application of the multilinear Strichartz estimate \eqref{Multilinear Strichartz} in the same way one proves local well posedness and similar to the way we showed smoothing for Case C of Lemma \ref{Lemma: HH & R2 regularity}, we get:
\[
\|\mathcal{HL}_B[z,v, \cdots, v]\|_{X^{s+\varepsilon, -1/2+}_T}\lesssim  \|z\|_{X^{s+\varepsilon, 1/2-}_T}\|v\|_{X^{\frac{p-5}{2(p-1)}+,1/2+}_T}^{p-1},
\]
which is respectable and enough for any $\varepsilon \in \mathbb{R}$. Notice that we have placed the entire $s+\varepsilon$ weight on $z$, with a $\delta$ power of $T$ following from localization.
\end{proof}

\begin{lemma}\label{Lemma: T composite in XSB Regularity}
Let $p\geq 5$ be odd, $s > \frac{p-3}{2(p-1)}$, and $0 < T\ll 1$. Then there is $\delta > 0$ so that
\begin{align*}
\|T[W_tu_0, v, &\left(\mathcal{R}^2+\mathcal{NR}\right)[v, \cdots, v], v, \cdots, v]\|_{X^{s+\varepsilon, -1/2+}_T}\\
&+\|T[ W_tu_0, \left(\mathcal{R}^2+\mathcal{NR}\right)[v, \cdots, v],v, \cdots, v]\|_{X^{s, -1/2+}_T}\lesssim  T^\delta \|u_0\|_{H^s_x}\|v\|_{X^{s, 1/2+}_T}^{2p-2},
\end{align*}
for
\[
0 < \varepsilon < 1.
\]
\end{lemma}
\begin{proof}
Both of these terms are handled the same way. We turn both of these multilinear operators into $2p-1$ multilinear operators as in Lemma \ref{Lemma: HL[T] Regularity}. Then, by ignoring any other structure in the frequency space associated to these estimates and using
\[
\frac{1}{|\Phi|}\lesssim \frac{1}{|k|}
\] 
to cancel the $\langle k\rangle^{\varepsilon}$ factor, it suffices to show the bound
\[
\left|\underset{\Gamma_{2p-1}}{\int\sum}\overline{\widehat{\eta}}_k\prod_{\substack{j=1\\odd}}^{2p-1}\widehat{v}_{k_j}\prod_{\substack{j=2\\even}}^{2p-1}\overline{\widehat{v}}_{k_j}\,d\Gamma\right|\lesssim \|\eta\|_{X^{-s,1/2-}}\|v\|_{X^{s,1/2-}_T}\|v\|^{2p-2}_{X^{s,1/2+}_T}.
\]
This bound, however, is simply the estimate required for local well-posedness of the $2p-1$ NLS, and hence by the multilinear Strichartz estimate \eqref{Multilinear Strichartz} we get that $0 < \varepsilon < 1$ is admissible for $s > \frac{p-3}{2(p-1)}.$ 
\end{proof}
We now combine the lemmas together in order to prove the first smoothing result.
\begin{proof}[Proof of Theorem \ref{Theorem 1: Smoothing}]
$ $\newline
Let $s > \frac{p-3}{2(p-1)}$,  $\chi(t)\in\mathcal{S}(\mathbb{R})$ as defined earlier, $b>\frac{1}{2}$ sufficiently close to $\frac{1}{2}$, and $T$ the local well-posedness time for $v$, a solution to \eqref{v equation} emanating from $u_0$. Define
\begin{align}
    \mathfrak{F}(v) &= \mathcal{E}[v, \cdots, v]+ \mathcal{HL}_B[T[W_t u_0, v, \cdots, v], v, \cdots, v]\\
    &\qquad-  \frac{p-1}{2}T[W_tu_0, v, \left(\mathcal{R}^2+\mathcal{NR}\right)[v, \cdots, v],v, \cdots, v]\nonumber\\
&\qquad+ \frac{p-1}{2}T[ W_tu_0, \left(\mathcal{R}^2+\mathcal{NR}\right)[v, \cdots, v], v, \cdots, v]\nonumber,
\end{align}
for $0 < t < \tilde{T} < T$, and notice that Duhamel form manifests as
\begin{align*}
    \Gamma[z]&= \chi(t/\tilde{T})W_tT[u_0] \pm i \chi(t/\tilde{T})\int_0^tW_{t-s}\left(\mathfrak{F}(v)+\mathcal{HL}_B[z, v, \cdots, v]\right)\,ds.
\end{align*}
Applying Lemmas \ref{Lemma: Normal Form Regularity}, \ref{Lemma: HH & R2 regularity}, \ref{Lemma: HL[T] Regularity}, \ref{Lemma: HL[z] regularity}, and \ref{Lemma: T composite in XSB Regularity} and a standard limiting argument immediately gives
\begin{align}\label{z bound}
    \|z\|_{X^{s+\varepsilon, b}_T} \lesssim_{\varepsilon} \|u_0\|_{H^{s}_x}^p +\tilde{T}^\delta(& \|v\|_{X^{s,b}_T}^p + \|v\|_{X^{s, b}_T}^{2p-1} + \|u_0\|_{H^s_x}\|v\|_{X^{s, b}_T}^{2p-2})\\
    &+\tilde{T}^\delta \|z\|_{X^{s+\varepsilon, b}_T}\|v\|^{2p-2}_{X^{s,b}_T}.\label{Equation: Smoothing term to subtract}
\end{align}
We then invoke the bound from Lemma \ref{Lemma: v local well posedness} and notice that for $\tilde{T} = \tilde{T}(\|u_0\|_{H^s_x}, T),$ \eqref{Equation: Smoothing term to subtract} has a coefficient less than $1$. Subtracting and taking $\tilde{T}$ smaller if necessary establishes the smoothing bound for $z$.

With the substitution in mind, we add and subtract $T$ and apply the triangle inequality to we find that (for $|t| < \tilde{T}$)
\[
\|v-W_tu_0\|_{C^0_tH^{s+\varepsilon}_x}\lesssim \|z\|_{C^0_tH^{s+\varepsilon}_x}+\|T[W_tu_0, v, \cdots, v]\|_{C^0_TH^{s+\varepsilon}_x}\leq C(\|u_0\|_{H^s_x}),
\]
by Lemma \ref{Lemma: Normal Form Regularity}, \eqref{z definition}, and Lemma \ref{z bound}, where $v = \mathcal{F}^{-1}_x\left(L_t[u]\widehat{u}\right).$
\end{proof}
\section{Global Attractors}\label{Section: Global Attractors}
For the remainder of the paper we let $f\in H^{1}_x(\mathbb{T})$ be a time-independent forcing function, $\gamma > 0$, and $W_t^\gamma$ denote the propagator of the modified linear group
\[
iu_t+\bigtriangleup u + i\gamma u = 0,
\]
given by 
\[
\widehat{W_t^\gamma u_0} = e^{-t(ik^2+\gamma)}\widehat{u_0}. 
\]

The next lemma records several facts that will be used implicitly in the proof.

\begin{lemma}\label{Lemma: modified linear embeddings}
Let $0 < T\ll 1$, $\gamma > 0$, $b =1/2+$, $p>2$ odd, $u_0\in H^s_x(\mathbb{T})$, $u, v\in X^{s,b}_T$, $F\in X^{s,b-1}_T$, and $f\in H^{s+1}_x(\mathbb{T})$ be time-independent. Then there is $\theta > 0$ so that the following hold
\begin{itemize}
    \item[1)] $\|W_t^\gamma u_0\|_{X^{s,b}_T}\lesssim_{\gamma}\|u_0\|_{H^s_x(\mathbb{T})},$
    \item[2)] $\|\mathcal{F}^{-1}_x(L_t[u]\widehat{f})\|_{X^{s, b-1}_T}\lesssim_\gamma T^{\theta}\|f\|_{H^s_x(\mathbb{T})},$
    \item[3)] $\|\mathcal{F}^{-1}_x(L_t[u]\widehat{f})\|_{X^{s, 1/2-}_T}\lesssim_\gamma \langle \|u\|_{X^{s,b}_T}\rangle^{p-1} \|f\|_{H^{s+1}_x(\mathbb{T})}$
    \item[4)] $\|\mathcal{F}^{-1}_x(L_t[u]\widehat{f}-L_t[v]\widehat{f})\|_{X^{s, b-1}_T}\lesssim_\gamma T^\theta \langle \|u\|_{X^{s,b}_T}+\|v\|_{X^{s,b}_T}\rangle^{p-2}\|u-v\|_{X^{s,b}_T}\|f\|_{H^s_x(\mathbb{T})}$
    \item[5)] $\left\|\int_0^t W_{t-s}^\gamma F(x,s)\,ds\right\|_{X^{s,b}_T}\lesssim_\gamma \|F\|_{X^{s,b-1}_T}$
\end{itemize}
\end{lemma}
\begin{proof}
These are all straightforward, but we include several of them for completeness. Let $\eta\in\mathcal{S}(\mathbb{R})$ be $1$ on $[-1,1]$ and $0$ outside $[-2,2]$. The first follows from the $H^s_xH^b_t$ characterization:
\begin{equation*}
    \|\eta(t/T)W_t^\gamma u_0\|_{X^{s,b}} = \|W_{-t}W_t^\gamma\eta(t/T) u_0\|_{H^s_xH^b_t} = \|e^{-t\gamma}\eta(t/T)\|_{H^b_t}\|u_0\|_{H^s_x}\lesssim \|u_0\|_{H^s_x}.
\end{equation*}

The second follows by time-localization:
\begin{align*}
    \|\mathcal{F}^{-1}(L_t[u]\widehat{f})\|_{X^{s, b-1}_T}&\lesssim T^{1/2-}\|\eta(t/T)\mathcal{F}^{-1}(L_t[u]\widehat{f})\|_{X^{s, 0}}\\
    &= T^{1/2-}\left\|\langle k\rangle^s\|e^{-itk^2}L_t[u]\eta(t/T)\|_{L^2_t}\widehat{f}\right\|_{\ell^2_k}\\
    &\lesssim T^{1-}\|f\|_{H^s_x}.
\end{align*}

The third follows from interpolation:
\begin{align}\label{forcing term in X^sb inequality}
    \|\mathcal{F}^{-1}(L_t[u]\widehat{f})\|_{X^{s, 1/2-}_T}&\lesssim \|\eta(t/T)\mathcal{F}^{-1}(L_t[u]\widehat{f})\|_{X^{s, 1/2-}}\nonumber\\
    &= \left\|\langle k\rangle^s\|e^{-itk^2}L_t[u]\eta(t/T)\|_{H^{1/2-}_t}\widehat{f}\right\|_{\ell^2_k},
\end{align}
but
\begin{align*}
    \|e^{-itk^2}L_t[u]\eta(t/T)\|_{L^2}&\lesssim T^{1/2}\\
    \|e^{-itk^2}L_t[u]\eta(t/T)\|_{H^{1}}&\lesssim \langle k\rangle^2\langle\|u^{p-1}(x,T)\eta(t/T)\|_{L^2_tL^1_x}\rangle T^{1/2} + T^{-1/2}\\
    &\lesssim \langle k\rangle^2\langle \|u\|_{X^{s,b}_T}\rangle^{p-1}T^{1/2}+T^{-1/2}.
\end{align*}
by Sobolev embedding, and standard Strichartz estimates for the $\|u\|_{L^{p-1}_x}^{p-1}$ term. Noting that the interpolation takes more of the $L^2_x$ bound, we see that substitution gives the desired bound.

The fourth follows similarly to the second. Indeed, 
\begin{align}\label{Forcing Term Contraction}
    \|\mathcal{F}^{-1}_x(L_t[u]\widehat{f}-&L_t[v]\widehat{f})\|_{X^{s, b-1}_T}\lesssim T^{1/2-}\|\mathcal{F}^{-1}_x(L_t[u]\widehat{f}-L_t[v]\widehat{f})\|_{X^{s, 0}_T}\nonumber\\
    &\leq T^{1/2-}\left\|\langle k\rangle^s\left\| \eta(t/T)\left(e^{i\frac{p-1}{2}\int_0^t\int_\mathbb{T}|u|^{p-1}-|v|^{p-1}\,dxds}-1\right)\right\|_{L^2_t}\widehat{f}\right\|_{\ell^2_k},
\end{align}
but
\begin{align*}
&\left\| \eta(t/T)\left(e^{i\frac{p-1}{2}\int_0^t\int_\mathbb{T}|u|^{p-1}-|v|^{p-1}\,dxds}-1\right)\right\|_{L^2_t}\lesssim \left\|\eta(t/T)\int_0^t\int_\mathbb{T}|u|^{p-1}-|v|^{p-1}\,dxds\right\|_{L^2_t}\\
&\qquad\qquad\qquad\qquad\lesssim T^{1-}\|u-v\|_{X^{s,b}_T}\max\left(\|u\|^{p-2}_{X^{s,b}_T}, \|v\|^{p-2}_{X^{s,b}_T}, \|u\|_{X^{s,b}_T}, \|v\|_{X^{s,b}_T}\right).
\end{align*}
Substitution back into \eqref{Forcing Term Contraction} completes the proof.

The last estimate is standard and can be found in (\cite{erdougan2016dispersive}, Lemma 4.6).
\end{proof}
\begin{remark}
Part 3 of Lemma \ref{Lemma: modified linear embeddings} shows that we must be careful in using our prior estimates with our forcing term. In particular, we have to use the fact that we're well above the Sobolev embedding level in order to avoid placing $f$ in $X^{1,1/2+}_T$.
\end{remark}

The next lemma we record is an a priori bound on the energy in order to establish the existence of an absorbing set in $H^1_x(\mathbb{T})$.
\begin{lemma}\label{Lemma: A priori Bound -- Absorbing Set}
Let $p \geq 5$ odd,  $\gamma > 0$, $u_0\in H^1_x(\mathbb{T})$, $f\in H^1_x(\mathbb{T})$ time independent, and $u$ the solution to
\begin{align*}
    \begin{cases}
        iu_t +\bigtriangleup u - |u|^{p-1}u +i\gamma u= f\\
        u(x,0) = u_0\in H^1(\mathbb{T}),
    \end{cases}
\end{align*}
emanating from $u_0$. Then $\|u\|_{L^2_x(\mathbb{T})}$ and $\|u\|_{\dot{H}^1_x(\mathbb{T})}$ are bounded \textit{a priori}. That is
\begin{align*}
    \limsup_{t\to\infty}\|u\|_{L^2_x(\mathbb{T})}&\leq C(\gamma, \|f\|_{L^2_x(\mathbb{T})}) \mbox{ and, }\\
    \limsup_{t\to\infty}\|u\|_{H^1_x(\mathbb{T})}&\leq C(\gamma, \|f\|_{H^1_x(\mathbb{T})}).
\end{align*}
\end{lemma}
\begin{proof}
By a standard limiting argument we assume that $u$ is smooth. We then first look at the $L^2_x(\mathbb{T})$ norm and differentiate with respect to time
\begin{align}
    \partial_t\frac{1}{2}\|u\|_{L^2(\mathbb{T})}^2 &= \frac{1}{2}\int_\mathbb{T}\overline{u}\left(iu_xx - i|u|^{p-1}u-\gamma u +if\right) + u\left(-i\overline{u}_xx+i|u|^{p-1}\overline{u}-\gamma \overline{u}-i\overline{f}\right)\,dx\nonumber\\
    &= -\gamma \|u\|_{L^2_x(\mathbb{T})}^2 + \int_\mathbb{T}\Im(u\overline{f})\,dx\label{L2 initial bound}
\end{align}
by standard conservation of mass for the \eqref{p-NLS}. Applying Young's inequality we get that
\begin{equation*}
\eqref{L2 initial bound}\leq -\frac{\gamma}{2} \|u\|_{L^2_x(\mathbb{T})}^2 + C(\gamma, \|f\|_{L^2_x(\mathbb{T})}),
\end{equation*}
and hence Gronwall's gives that $\limsup_{t\to\infty}\|u\|_{L^2_x(\mathbb{T})}\leq C(\gamma, \|f\|_{L^2_x(\mathbb{T})}).$

Similarly, we differentiate \eqref{energy} and simplify using conservation of energy for \eqref{p-NLS}. This leaves us with only terms involving $\gamma$ and the forcing term, $f$:
\begin{align}
\partial_t\bigg(\frac{1}{2}\|u_x&\|_{L^2_x(\mathbb{T})}^2 +\frac{1}{p+1}\|u\|_{L^{p+1}_x(\mathbb{T})}^{p+1}\bigg) = -\gamma\left( \|u_x\|_{L^2_x(\mathbb{T})}^2 + \|u\|_{L^{p+1}_x(\mathbb{T})}^{p+1}\right) - \frac{i}{2}\int_\mathbb{T} u_x\overline{f}_x-\overline{u}_xf_x\,dx\nonumber\\
&\,\,\,\qquad-\frac{i}{2}\int_\mathbb{T} u^{\frac{p+1}{2}}\overline{u}^{\frac{p+1}{2}-1}\overline{f}-\overline{u}^{\frac{p+1}{2}}u^{\frac{p+1}{2}-1}f\,dx\nonumber\\
&\leq -2\gamma \left(\frac{1}{2}\|u_x\|_{L^2_x(\mathbb{T})}^2 +\frac{1}{p+1}\|u\|_{L^{p+1}_x(\mathbb{T})}^{p+1}\right) + \int_\mathbb{T} \Im(u_x\overline{f})+\Im(u^{\frac{p+1}{2}}\overline{u}^{\frac{p+1}{2}-1}\overline{f})\,dx\label{ux initial bound}.
\end{align}
We again apply Young's to get
\begin{align*}
\eqref{ux initial bound}&\leq -\gamma \left(\frac{1}{2}\|u_x\|_{L^2_x(\mathbb{T})}^2 +\frac{1}{p+1}\|u\|_{L^{p+1}_x(\mathbb{T})}^{p+1}\right) + C(\gamma, \|f_x\|_{L^2_x(\mathbb{T})}) + C(\gamma, \|f\|_{L^{p+1}_x(\mathbb{T})}))\\
&\leq-\gamma \left(\frac{1}{2}\|u_x\|_{L^2_x(\mathbb{T})}^2 +\frac{1}{p+1}\|u\|_{L^{p+1}_x(\mathbb{T})}^{p+1}\right) + C(\gamma, \|f\|_{H^1_x(\mathbb{T})}),
\end{align*}
by Sobolev embedding. Gronwall's then gives that
\[
\limsup_{t\to\infty} \frac{1}{2}\|u_x\|_{L^2_x(\mathbb{T})}\leq \limsup_{t\to\infty}\left(\frac{1}{2}\|u_x\|_{L^2_x(\mathbb{T})}^2 +\frac{1}{p+1}\|u\|_{L^{p+1}_x(\mathbb{T})}^{p+1}\right)\leq C(\gamma, \|f\|_{H^1_x(\mathbb{T})}).
\]

Combining the $L^2_x(\mathbb{T})$ and the $\dot{H}^1_x(\mathbb{T})$ bounds we conclude that 
\[
\limsup_{t\to\infty}\|u\|_{H^1_x(\mathbb{T})}\leq C(\gamma, \|f\|_{H^1_x(\mathbb{T})}),
\]
as desired.
\end{proof}
\subsection{Reductions and Lemmas}

We can rewrite \eqref{p-fdNLS} using the transformation $L_t[u]$ from $\eqref{L definition}$ into
\begin{align}\label{v fdnls}
    \begin{cases}
        iv_t +\bigtriangleup v +i\gamma v- \mathcal{R}^2[v, \cdots, v]-\mathcal{NR}[v, \cdots, v]= F\\
        v(x,0) = u_0\in H^1_x(\mathbb{T}),
    \end{cases}
\end{align}
where 
\[
F(v, x):=\mathcal{F}^{-1}_x(L_t[v]\widehat{f}).
\]
Note that the forcing term is no longer time-independent, but depends in a nice enough way on the solution $v$. However, we retain the following corollary whose major contribution is the local well-posedness bound \eqref{Forced v Local Wellposedness}.

\begin{corollary}
Let $u_0\in H^1_x(\mathbb{T})$, $f\in H^1_x(\mathbb{T})$ time independent, and $F(v,x) = \mathcal{F}^{-1}_x(L_t[v]\widehat{f}).$ Then there is a
\[
T = T(\|u_0\|_{H^1_x(\mathbb{T})}, \|f\|_{H^{1}_x(\mathbb{T})}, \gamma)
\]
so that on $0 < t < T$ there is a unique solution to \eqref{v fdnls} in $X^{1, 1/2+}$ emanating from $u_0$ that satisfies the bound
\begin{equation}\label{Forced v Local Wellposedness}
\|v\|_{X^{1,1/2+}_T}\lesssim C(\|u_0\|_{H^1_x(\mathbb{T})}, \|f\|_{H^{1}_x(\mathbb{T})}, \gamma).
\end{equation}
\end{corollary}
\begin{proof}
This is standard and follows from Sobolev embedding and Lemma \ref{Lemma: modified linear embeddings}.
\end{proof}

Note that we will not be able to place $F$ into $H^{s+\varepsilon}_x$. This technicality necessitates an additional term in the change of variables \eqref{z definition}, which results in additional forcing terms.

We now recall $T$ from \eqref{defn of T} and perform the substitution 
\begin{equation}\label{Definition: z forced}
    v = W_t^\gamma u_0 - T[W_t^\gamma u_0, v, \cdots, v]+G+z,
\end{equation}
where 
\begin{align}
    (i\partial_t +\partial_x^2 +i\gamma)T[W_t^\gamma &u_0, v, \cdots, v] = -\widehat{\mathcal{HL}_B}[W_t^\gamma u_0, v \cdots, v]\\
&\qquad- \frac{p-1}{2}\widehat{T}[W_t^\gamma u_0, (i\partial_t+\bigtriangleup)v,v, \cdots, v]\nonumber\\
&\qquad+\frac{p-1}{2}\widehat{T}[W_t^\gamma u_0, v, (i\partial_t+\bigtriangleup)v,v, \cdots, v],\nonumber
\end{align}
by \eqref{Calculation: T in Equation}, and
\[
G(v,x) = \frac{F}{\bigtriangleup+i\gamma}.
\]

We find 
\[
(i\partial_t+\partial_x^2+i\gamma)G = \frac{p+1}{2}\|v\|_{L_x^{p-1}}^{p-1}G+F,
\]
and hence, analogously to \eqref{Equation}, $z$ satisfies (by substitution into the linear and high-frequency portion of $\mathcal{HL_B}$)
\begin{align}\label{Global Attractor Equation}
    \begin{cases}
        iz_t +\bigtriangleup z + i\gamma z -\mathcal{E}[v, \cdots, v]\\ \,\,\,+\mathcal{HL}_B[T[W_t^\gamma u_0, v, \cdots, v]- G -z , v, \cdots, v]\\ 
        \,\,\,\,\,\,+  \frac{p-1}{2}T[W_t^\gamma u_0, v, \left(\mathcal{R}^2+\mathcal{NR}\right)[v, \cdots, v]-i\gamma v+F,v \cdots, v]\\
\,\,\,\,\,\,\,\,\,- \frac{p-1}{2}T[ W_t^\gamma u_0, \left(\mathcal{R}^2+\mathcal{NR}\right)[v, \cdots, v]-i\gamma v + F,v, \cdots, v]= \frac{p+1}{2}\|v\|_{L^{p-1}_x}^{p-1} G\\
        z(x,0) = T[u_0, \cdots, u_0]-\frac{f}{\bigtriangleup+i\gamma}\in H^s_x(\mathbb{T}).
    \end{cases}
\end{align}

By the first bound of lemma \ref{Lemma: modified linear embeddings} we see every all the prior lemmas hold. We then restrict ourselves to estimating the other new terms from \eqref{Global Attractor Equation} below.
\begin{lemma}\label{Lemma: gamma and F in T}
Let $p\geq 5$ be odd and $0 < T\ll 1.$ Then there is some $\delta > 0$
\begin{align*}
\|T[W_t^\gamma u_0, F, v, \cdots, v]\|_{X^{1+\varepsilon, -1/2+}_T}&\lesssim_\delta T^\delta \|u_0\|_{H^1_x(\mathbb{T})}\|f\|_{H^{1}_x(\mathbb{T})}\|v\|_{X^{1,1/2+}_T}^{p-2}\\
\|\gamma T[W_t^\gamma u_0, v, \cdots, v]\|_{X^{1+\varepsilon, -1/2+}_T}&\lesssim_\delta T^\delta \gamma \|u_0\|_{H^1_x(\mathbb{T})}\|v\|^{p-1}_{X^{1,1/2+}_T}
\end{align*}
for $0 < \varepsilon < 1$. 
\end{lemma}
\begin{proof}
Since $T$ comes with the restriction that $|\Phi|\gtrsim |k|$, the estimate
\begin{equation}\label{Equation: F in T bound}
\|T[W_t^\gamma u_0, F, v, \cdots v]\|_{X^{1+\varepsilon, -1/2+}_T}\lesssim_\delta T^\delta \|u_0\|_{H^1_x(\mathbb{T})}\|F\|_{C^0_tH^s_x}\|v\|_{X^{1,1/2+}_T}^{p-2}
\end{equation}
follows immediately from Sobolev embedding.

The result then follows by noting that $L_t$ preserves the $H^s$ norm:
\[
\eqref{Equation: F in T bound}\lesssim T^\delta \|u_0\|_{H^1_x(\mathbb{T})}\|f\|_{H^{1}_x(\mathbb{T})}\|v\|_{X^{1,1/2+}_T}^{p-2}.
\]
The second estimate follows similarly, but for much lower regularity by the multilinear Strichartz estimate \eqref{Multilinear Strichartz}.
\end{proof}
\begin{lemma}\label{Lemma: G Estimates}
Let $p\geq 5$ be odd and $0 < T\ll 1.$ Then there is some $\delta > 0$
\begin{align*}
\|\mathcal{HL}_B[G, v, \cdots, v]\|_{X^{1+\varepsilon, -1/2+}_T}&\lesssim_\delta T^\delta\|f\|_{H^{1}_x(\mathbb{T})}\|v\|_{X^{1,1/2+}_T}^{p-1}\\
\left\|\|v\|_{L^{p-1}_x}^{p-1}G\right\|_{X^{1+\varepsilon, -1/2+}_T}&\lesssim_\delta T^\delta\|f\|_{H^1_x}\|v\|^{p-1}_{X^{1,1/2+}_T}
\end{align*}
for $0 < \varepsilon < 1$. 
\end{lemma}
\begin{proof}
We only prove the first estimate as the other follows similarly.

Since $G$ comes with two derivative savings and is in the high-frequency component, we may use time localization and Sobolev embedding to get
\begin{align*}
    \|\mathcal{HL}_B[G, v, \cdots, v]\|_{X^{1+\varepsilon, -1/2+}_T}\lesssim T^\delta \|\mathcal{HL}_B[G, v, \cdots, v]\|_{X^{1+\varepsilon,0}_T}\lesssim \|G\|_{C^0_tH^{1+\varepsilon}_x}\|v\|_{X^{1,1/2+}_T}^{p-1}.
\end{align*}
This is good for $0 < \varepsilon < 2$.

\end{proof}

We now prove, using the lemmas in the prior sections, that the new $v$ has non-linear smoothing.
\begin{lemma}\label{Lemma: Smoothing for forced and damped equation}
Let $p\geq 5$ odd, $0 < \varepsilon < \frac{p+3}{2(p-1)}$, $\gamma > 0$, $v$ the solution to \eqref{p-fdNLS} emanating from $u_0\in H^1_x(\mathbb{T})$ with lifespan $0<T\ll 1$, $F = \mathcal{F}_x^{-1}(L_t[v]\widehat{f})$ for $f\in H^{1}_x(\mathbb{T})$ time independent. Then there is a $\tilde{T} = \tilde{T}(\varepsilon, \gamma, \|f\|_{H^1_x(\mathbb{T})}, \|u_0\|_{H^1_x(\mathbb{T})})$ so that for $0 < t < \min(\tilde{T}, T)$ the following estimate holds
\[
\|v - W_t^\gamma u_0\|_{H^{1+\varepsilon}}\leq C(\varepsilon, \gamma, \|f\|_{H^{1}_x(\mathbb{T})}, \|u_0\|_{H^1_x(\mathbb{T})}).
\]
Moreover, there is a $\tilde{C}= \tilde{C}(\varepsilon, \gamma, \|f\|_{H^{1}_x(\mathbb{T})}, \|u_0\|_{H^1_x(\mathbb{T})})$ so that for all time $t>0$ the bound
\begin{equation}\label{Equation: Forced Smoothing Global Bound}
\|v - W_t^\gamma u_0\|_{H_x^{1+\varepsilon}}\leq \tilde{C}
\end{equation}
holds.
\end{lemma}
    
\begin{proof}
Similar to the smoothing proof for the unforced and undamped equation, we let
\begin{align}
    &\mathfrak{F}(v) = \mathcal{E}[v, \cdots, v]- \mathcal{HL}_B[T[W_t^\gamma u_0, v, \cdots, v]-G(v,x),v,\cdots, v] \nonumber\\
    &\quad- \frac{p-1}{2}T[W_t^\gamma u_0, v, \left(\mathcal{R}^2+\mathcal{NR}\right)[v, \cdots, v],v, \cdots, v]\nonumber\\
&\quad+ \frac{p-1}{2}T[ W_t^\gamma u_0, \left(\mathcal{R}^2+\mathcal{NR}\right)[v, \cdots, v],v, \cdots, v]\nonumber\\
&\quad + \frac{p-1}{2}T[W_t^\gamma u_0, F(v, x), v, \cdots, v]\nonumber\\
&\quad- \frac{p-1}{2}T[W_t^\gamma u_0, v, F(v, x), v, \cdots, v]\nonumber\\
&\quad -\frac{p+1}{2}\|v\|_{L^{p-1}_x}^{p-1}G(v,x) + i(p-1)\gamma T[W_t^\gamma u_0, v, \cdots, v]\nonumber.
\end{align}
Recalling \eqref{Global Attractor Equation}, it follows by Duhamel that $z$ satisfies
\begin{align*}
z &= \chi(t/\tilde{T})\left(W_t^\gamma T[u_0, \cdots, u_0]-\frac{f}{\bigtriangleup+i\gamma}\right)\\
&\qquad-i\chi(t/\tilde{T})\int_0^tW_{t-s}^\gamma\left(\mathfrak{F}(v) + \mathcal{HL}_B[z, v, \cdots, v]\right)\,ds.
\end{align*}
Applying Lemmas \ref{Lemma: Normal Form Regularity}, \ref{Lemma: HH & R2 regularity}, \ref{Lemma: HL[T] Regularity}, \ref{Lemma: HL[z] regularity}, \ref{Lemma: T composite in XSB Regularity}, \ref{Lemma: gamma and F in T}, and \ref{Lemma: G Estimates} immediately gives
\begin{align*}
    \|z\|_{X^{1+\varepsilon, 1/2+}_T} &\lesssim_{\varepsilon,\gamma} \|u_0\|_{H^{1}_x}^p + \tilde{T}^\delta(\|v\|_{X^{1,1/2+}_T}^p + \|v\|_{X^{1, 1/2+}_T}^{2p-1} + \|u_0\|_{H^s_x}\|v\|_{X^{1, 1/2+}_T}^{2p-2})\\
    &+\tilde{T}^\delta \|u_0\|_{H^1_x}\|f\|_{H^{1}_x}\|v\|^{p-2}_{X^{1,1/2+}_T}\\
    &+ \tilde{T}^\delta(\|v\|^{p-1}_{X^{1,1/2+}_T}+\gamma)\|f\|_{H^1_x}+\tilde{T}^\delta \|z\|_{X^{1+\varepsilon, 1/2+}_T}\|v\|^{2p-2}_{X^{1,1/2+}_T}.
\end{align*}
Applying the local well-posedness bound \eqref{Forced v Local Wellposedness}, we find
\begin{align}
    \|z\|_{X^{1+\varepsilon, 1/2+}_T}&\leq \tilde{T}^\delta C_1(\|u_0\|_{H^1_x(\mathbb{T})}, \|f\|_{H_x^{1}(\mathbb{T})}, \gamma, \varepsilon)\label{Equation: Global Attractor z bound line}\\
    &\qquad+ \tilde{T}^\delta \|z\|_{X^{1+\varepsilon, 1/2+}_T}C_2(\|u_0\|_{H^1_x(\mathbb{T})}, \|f\|_{H^{1}_x(\mathbb{T})}, \gamma, \varepsilon)\label{Equation: Global Attractor T on Z line}.
\end{align}
Taking $\tilde{T} = \tilde{T}(T, C_2) = \tilde{T}(\|u_0\|_{H^1_x(\mathbb{T})}, \|f\|_{H^{1}_x(\mathbb{T})}, \gamma, \varepsilon)$ small enough allows us to subtract line \eqref{Equation: Global Attractor T on Z line} from line \eqref{Equation: Global Attractor z bound line} and obtain the bound
\[
\|z\|_{X^{1+\varepsilon, 1/2+}_T}\leq C(\|u_0\|_{H^1_x(\mathbb{T})}, \|f\|_{H^{1}_x(\mathbb{T})}, \gamma, \varepsilon).
\]

We now recall the definition of $z$ on line \eqref{Definition: z forced}, and use the triangle inequality with Lemma \ref{Lemma: Normal Form Regularity} to conclude the desired bound
\[
\|v - W_t^\gamma u_0\|_{C^0_tH^{1+\varepsilon}_x}\lesssim\|v - W_t^\gamma u_0\|_{X^{1+\varepsilon, 1/2+}_T}\leq C(\|u_0\|_{H^1_x(\mathbb{T})}, \|f\|_{H^{1}_x(\mathbb{T})}, \gamma, \varepsilon),
\]
for $0 < \varepsilon < \frac{p+3}{2(p-1)}$.

All that remains is to prove that this bound extends globally. 
By line \eqref{L definition} and Lemma \ref{Lemma: A priori Bound -- Absorbing Set} we know that for all time $t > 0$, the bound
\[
\|v(t)\|_{H^{1}_x(\mathbb{T})} = \|u(t)\|_{H^{1}_x(\mathbb{T})}\leq C(\|u_0\|_{H^1_x(\mathbb{T})}, \|f\|_{H^{1}_x(\mathbb{T})}, \gamma)
\]
holds and that $u$ exists globally. From this, it follows that for the local non-linear smoothing time $\tilde{T}$, $n\in\mathbb{N}$, and $t\in [n\tilde{T}, (n+1)\tilde{T}]$ that
\begin{align*}
\left\|v(t) - W_{t-n\tilde{T}}^\gamma v(n\tilde{T})\right\|_{H^{1+\varepsilon}_x(\mathbb{T})}&\lesssim C(\|v(n\tilde{T})\|_{H^1_x(\mathbb{T})}, \|f\|_{H^{1}_x(\mathbb{T})}, \gamma, \varepsilon)\\
&\leq C(\|u_0\|_{H^1_x(\mathbb{T})}, \|f\|_{H^{1}_x(\mathbb{T})}, \gamma, \varepsilon).
\end{align*}
Then for $J\in\mathbb{N}$ and $t = J\tilde{T}$ we have
\begin{align*}
    \big\|v(J\tilde{T}) - W_t^\gamma u_0\big\|_{H^{1+\varepsilon}_x(\mathbb{T})}&\leq \sum_{j=1}^J \left\|W_{(J-j)\tilde{T}}^\gamma v(j\tilde{T}) - W_{(J-j+1)\tilde{T}}^\gamma v((j-1)\tilde{T})\right\|_{H_x^{1+\varepsilon}(\mathbb{T})}\\
    &\leq\sum_{j=1}^J e^{-(J-j)\tilde{T}\gamma}\left\|v(j\tilde{T}) - W^\gamma_{\tilde{T}}v((j-1)\tilde{T})\right\|_{H_x^{1+\varepsilon}(\mathbb{T})}\\
    &\lesssim C\sum_{j=0}^\infty e^{-j\tilde{T}\gamma}\\
    &= \tilde{C}(\|u_0\|_{H^1_x(\mathbb{T})}, \|f\|_{H^{1}_x(\mathbb{T})}, \gamma, \varepsilon),
\end{align*}
establishing \eqref{Equation: Forced Smoothing Global Bound}.
\end{proof}
All that remains is to prove Corollary \ref{Corollary: global Attractor}. To do so, we will follow a general strategy used to prove such statements in the presence of smoothing.
\begin{proof}[Proof of Corollary \ref{Corollary: global Attractor}]
Now, let $0 < \varepsilon < \frac{p+3}{2(p-1)},$ $S_t$ be the (global) data-to-solution map associated to \eqref{p-fdNLS} at time $t$, $L_t^{-1}[v]$ denote the Fourier multiplier\footnote{This is abuse of notation, but the benefits of this far outweigh any potential confusion that may arise.} with factor $\frac{1}{L_t[v]}$ and similarly for $L_t[v]$, $B$ the absorbing set (guaranteed by Lemma \ref{Lemma: A priori Bound -- Absorbing Set}), $u_0\in B$, and $N_t$ defined by the relation
\[
S_tu_0 = L_{t}^{-1}[S_tu_0]W^\gamma_tu_0+ N_tu_0.
\]
It's immediate that, uniformly on $B$, \begin{equation}\label{uniformDecay}
    \|L_{t}^{-1}[S_tu_0]W^\gamma_tu_0\|_{H^1_x(\mathbb{T})} = \|W^\gamma_t u_0\|_{H^1_x(\mathbb{T})} \lesssim e^{-\gamma t}\to 0.
\end{equation} 
Furthermore, by Lemma \ref{Lemma: Smoothing for forced and damped equation} we see that, on $B$, 
\begin{align*}
    \|N_tu_0\|_{H^{1+\varepsilon}(\mathbb{T})} & = \left\|\left(L_{t}[S_tu_0]S_t - W_t^\gamma\right) u_0\right\|_{H_x^{1+\varepsilon}(\mathbb{T})}\\
    &\leq \tilde{C}(\|f\|_{H^1_x(\mathbb{T})}, \gamma, \varepsilon),
\end{align*}
and thus by Rellich's Theorem, $\{N_tu_0\,:\,t > 0\}$ is pre-compact in $H^1_x(\mathbb{T})$. It follows that $S_t$ is asymptotically compact, ensuring the existence of a global attractor, A (For more information, see \cite{temam2012infinite}).

To prove the final claim we let $0 < \varepsilon < \frac{p+3}{2(p-1)}$, note that \[A = \omega(B) = \bigcap_{\tau} \overline{\bigcup_{t > \tau} S_t B}:= \bigcap_\tau U_\tau,\] and define $B_\varepsilon$ to be the ball of radius $C(\|f\|_{H^1_x(\mathbb{T})}, \gamma, \varepsilon)$ in $H_x^{1+\varepsilon/2+\frac{p+3}{4(p-1)}}$. By our proof of the existence of an absorbing set, there is a time $T = T(\gamma, \|u_0\|_{H^1_x(\mathbb{T})}, \|f\|_{H^1_x(\mathbb{T})})$ so that for $t \geq T$,
\begin{equation}\label{Equation: Solution Decay for Showing Compact}
\|S_t u_0\|_{H^1_x(\mathbb{T})}\leq C(\gamma, \|f\|_{H^1_x(\mathbb{T})}).
\end{equation}
Now, by Lemma \ref{Lemma: Smoothing for forced and damped equation}, the prior sentence, and translation in time, we assume that $S_tu_0$ is in the absorbing set for all $t > 0$ and satisfies \eqref{Equation: Solution Decay for Showing Compact}. Hence, by \eqref{uniformDecay} we find for $\tau > T$ that $U_\tau \subset B_\varepsilon + B(\delta_\tau, 0)$ in $H^1_x(\mathbb{T})$, where $\delta_\tau \to 0$ as $\tau\to \infty$. Since $B_\varepsilon$ is compact in $H^1_x(\mathbb{T})$, it follows that 
\[
A = \bigcap_\tau U_\tau\subset \bigcap_\tau \left(B_\varepsilon+B(\delta_\tau,0)\right) = B_\varepsilon,
\]
a compact set in $H^{1+\varepsilon}_x(\mathbb{T})$ by Rellich's theorem.
\end{proof}
\section{Acknowledgements}
The author would like to thank Professor Burak Erdo{\u{g}}an for bringing the connection between smoothing and global attractors to my attention, and for his endless patience.
\bibliographystyle{acm}
\bibliography{Waveguide, GeneralDispersivePDE, NLS, gKdV, Attractor}

\begin{thebibliography}{10}

\bibitem{akroune1999regularity}
{\sc Akroune, N.}
\newblock Regularity of the attractor for a weakly damped nonlinear
  {S}chr{\"o}dinger equation on $\mathbb{R}$.
\newblock {\em Applied Mathematics Letters 12}, 3 (1999), 45--48.

\bibitem{bourgain1993fourier}
{\sc Bourgain, J.}
\newblock Fourier transform restriction phenomena for certain lattice subsets
  and applications to nonlinear evolution equations.
\newblock {\em Geometric \& Functional Analysis GAFA 3}, 3 (1993), 209--262.

\bibitem{bourgain1998refinements}
{\sc Bourgain, J.}
\newblock Refinements of {S}trichartz inequality and applications to 2{D}-{NLS}
  with critical nonlinearity.
\newblock {\em International Mathematics Research Notices 1998}, 5 (1998),
  253--283.

\bibitem{bourgain2004remark}
{\sc Bourgain, J.}
\newblock A remark on normal forms and the “{I}-method” for periodic {NLS}.
\newblock {\em Journal d’Analyse Mathematique 94}, 1 (2004), 125--157.

\bibitem{cazenave1989some}
{\sc Cazenave, T., and Weissler, F.~B.}
\newblock {S}ome remarks on the nonlinear {S}chr{\"o}dinger equation in the
  critical case.
\newblock In {\em Nonlinear semigroups, partial differential equations and
  attractors}. Springer, 1989, pp.~18--29.

\bibitem{cazenave1990cauchy}
{\sc Cazenave, T., and Weissler, F.~B.}
\newblock The {C}auchy problem for the critical nonlinear {S}chr{\"o}dinger
  equation in ${H}^s$.
\newblock {\em Nonlinear Analysis: Theory, Methods \& Applications 14}, 10
  (1990), 807--836.

\bibitem{correia2020nonlinear}
{\sc Correia, S., and Silva, J.}
\newblock Nonlinear smoothing for dispersive {PDE}: a unified approach.
\newblock {\em Journal of Differential Equations\/} (2020).

\bibitem{de2007global}
{\sc {De Silva}, D., Pavlovi{\'c}, N., Staffilani, G., and Tzirakis, N.}
\newblock Global well-posedness for a periodic nonlinear {S}chr{\"o}dinger
  equation in 1{D} and 2{D}.
\newblock {\em Discrete and Continuous Dynamical Systems 19}, 1 (2007), 37--65.

\bibitem{dodson2016global}
{\sc Dodson, B.}
\newblock {G}lobal well-posedness and scattering for the defocusing, l
  2-critical, nonlinear {S}chr{\"o}dinger equation when d= 1.
\newblock {\em American Journal of Mathematics 138}, 2 (2016), 531--569.

\bibitem{erdougan2013global}
{\sc Erdo{\u{g}}an, M., and Tzirakis, N.}
\newblock Global smoothing for the periodic {K}d{V} evolution.
\newblock {\em International Mathematics Research Notices 2013}, 20 (2013),
  4589--4614.

\bibitem{erdougan2016dispersive}
{\sc Erdo{\u{g}}an, M., and Tzirakis, N.}
\newblock {\em Dispersive partial differential equations: wellposedness and
  applications}, vol.~86.
\newblock Cambridge University Press, 2016.

\bibitem{erdogan2017smoothing}
{\sc Erdo{\u{g}}an, M.~B., G{\"u}rel, T.~B., and Tzirakis, N.}
\newblock Smoothing for the fractional {S}chr{\"o}dinger equation on the torus
  and the real line.
\newblock {\em Indiana Univ. Math. J. 68\/} (2019), 369--392.

\bibitem{erdogan2011long}
{\sc Erdo{\u{g}}an, M.~B., and Tzirakis, N.}
\newblock Long time dynamics for forced and weakly damped {K}d{V} on the torus.
\newblock {\em Communications on Pure \& Applied Analysis 12}, 6 (2013).

\bibitem{erdougan2013smoothing}
{\sc Erdo{\u{g}}an, M.~B., and Tzirakis, N.}
\newblock Smoothing and global attractors for the {Z}akharov system on the
  torus.
\newblock {\em Analysis \& PDE 6}, 3 (2013), 723--750.

\bibitem{erdogan2013talbot}
{\sc Erdo{\u{g}}an, M.~B., and Tzirakis, N.}
\newblock Talbot effect for the cubic non-linear {S}chr{\"o}edinger equation on
  the torus.
\newblock {\em Mathematical Research Letters 20}, 6 (2013), 1081--1090.

\bibitem{fan2017log}
{\sc Fan, C.}
\newblock {L}og--log blow up solutions blow up at exactly m points.
\newblock In {\em Annales de l'Institut Henri Poincare (C) Non Linear
  Analysis\/} (2017), vol.~34, Elsevier, pp.~1429--1482.

\bibitem{ghidaglia1988finite}
{\sc Ghidaglia, J.-M.}
\newblock Finite dimensional behavior for weakly damped driven
  {S}chr{\"o}dinger equations.
\newblock In {\em Annales de l'Institut Henri Poincare (C) Non Linear
  Analysis\/} (1988), vol.~5, Elsevier, pp.~365--405.

\bibitem{goubet1996regularity}
{\sc Goubet, O.}
\newblock Regularity of the attractor for a weakly damped nonlinear
  {S}chr{\"o}dinger equation.
\newblock {\em Applicable Analysis 60}, 1-2 (1996), 99--119.

\bibitem{goubet2000asymptotic}
{\sc Goubet, O.}
\newblock Asymptotic smoothing effect for weakly damped forced {K}orteweg-de
  {V}ries equations.
\newblock {\em Discrete \& Continuous Dynamical Systems-A 6}, 3 (2000), 625.

\bibitem{goubet2017finite}
{\sc Goubet, O., and Zahrouni, E.}
\newblock Finite dimensional global attractor for a fractional nonlinear
  {S}chr{\"o}dinger equation.
\newblock {\em Nonlinear Differential Equations and Applications NoDEA 24}, 5
  (2017), 1--16.

\bibitem{goubet2020global}
{\sc Goubet, O., and Zahrouni, E.}
\newblock Global attractor for damped forced nonlinear logarithmic
  schr{\"o}dinger equations.
\newblock {\em Discrete \& Continuous Dynamical Systems-S\/} (2020).

\bibitem{isom2020growth}
{\sc Isom, B., Mantzavinos, D., and Stefanov, A.}
\newblock Growth bound and nonlinear smoothing for the periodic derivative
  nonlinear {S}chr{\"o}dinger equation.
\newblock {\em arXiv preprint arXiv:2012.09933\/} (2020).

\bibitem{kappeler2017scattering}
{\sc Kappeler, T., Schaad, B., and Topalov, P.}
\newblock Scattering-like phenomena of the periodic defocusing {NLS} equation.
\newblock {\em Mathematical Research Letters 24}, 3 (2017), 803--826.

\bibitem{keraani2009smoothing}
{\sc Keraani, S., and Vargas, A.}
\newblock A smoothing property for the ${L}^2$-critical {NLS} equations and an
  application to blowup theory.
\newblock {\em Annales de l'Institut Henri Poincar{\'e} (C) Non Linear Analysis
  26}, 3 (2009), 745--762.

\bibitem{kishimoto2014remark}
{\sc Kishimoto, N.}
\newblock {R}emark on the periodic mass critical nonlinear {S}chr{\"o}dinger
  equation.
\newblock {\em Proceedings of the American Mathematical Society 142}, 8 (2014),
  2649--2660.

\bibitem{li2011global}
{\sc Li, Y., Wu, Y., and Xu, G.}
\newblock Global well-posedness for the mass-critical nonlinear
  {S}chr{\"o}dinger equation on $\mathbb{T}$.
\newblock {\em Journal of Differential Equations 250}, 6 (2011), 2715--2736.

\bibitem{mcconnell2021global}
{\sc McConnell, R.}
\newblock {G}lobal attractor for the periodic generalized {K}orteweg-de {V}ries
  equation through smoothing.
\newblock {\em arXiv preprint arXiv:2105.13405\/} (2021).

\bibitem{molinet2009global}
{\sc Molinet, L.}
\newblock {G}lobal attractor and asymptotic smoothing effects for the weakly
  damped cubic {S}chr{\"o}dinger equation in ${L}^2(\mathbb{T})$.
\newblock {\em Dynamics of Partial Differential Equations 6}, 1 (2009), 15--34.

\bibitem{oh2020smoothing}
{\sc Oh, S., and Stefanov, A.}
\newblock Smoothing and growth bound of periodic generalized {K}orteweg-de
  {V}ries equation.
\newblock {\em arXiv preprint arXiv:2001.08984\/} (2020).

\bibitem{oh2012blowup}
{\sc Oh, T.}
\newblock A blowup result for the periodic {NLS} without gauge invariance.
\newblock {\em Comptes Rendus Mathematique 350}, 7-8 (2012), 389--392.

\bibitem{tao2008global}
{\sc Tao, T.}
\newblock A global compact attractor for high-dimensional defocusing non-linear
  {S}chr{\"o}dinger equations with potential.
\newblock {\em Dynamics of Partial Differential Equations 5}, 2 (2008),
  101--116.

\bibitem{temam2012infinite}
{\sc Temam, R.}
\newblock {\em Infinite-dimensional dynamical systems in mechanics and
  physics}, vol.~68.
\newblock Springer Science \& Business Media, 2012.

\bibitem{wang1995energy}
{\sc Wang, X.}
\newblock {A}n energy equation for the weakly damped driven nonlinear
  {S}chr{\"o}dinger equations and its application to their attractors.
\newblock {\em Physica D: Nonlinear Phenomena 88}, 3-4 (1995), 167--175.

\bibitem{wang2013periodic}
{\sc Wang, Y.}
\newblock Periodic {N}onlinear {S}chr{\"o}dinger {E}quation in {C}ritical
  ${H}^s(\mathbb{T}^n)$ {S}paces.
\newblock {\em SIAM Journal on Mathematical Analysis 45}, 3 (2013), 1691--1703.

\end{thebibliography}
\email{ryanm12@illinois.edu}
\end{document}